\documentclass[final,floating,letterpaper]{siamltex1213}
\usepackage{}
\usepackage{latexsym,amsmath,amssymb}
\allowdisplaybreaks[4]

\usepackage{comment}
\usepackage[compress]{cite}
\usepackage{algorithm}
\usepackage{algorithmic}

\algsetup{indent=2em}

\usepackage{mathbbold}
\usepackage{bbm}
\usepackage{amsfonts}
\usepackage[T1]{fontenc}
\usepackage{latexsym,amssymb,amsmath,amsfonts}
\usepackage{graphicx}
\usepackage{subfigure}
\usepackage{epsf}
\usepackage{epstopdf}
\usepackage{amssymb}
\usepackage{hyperref}

\usepackage{bm}

\newtheorem{remark}[theorem]{Remark}

\renewcommand{\theequation}{\arabic{section}.\arabic{equation}}

\newcommand{\R}{{\mathbb R}}

\newcommand{\C}{{\mathbb C}}

\newcommand{\be}{\begin{eqnarray}}
\newcommand{\ben}{\begin{eqnarray*}}
\newcommand{\en}{\end{eqnarray}}
\newcommand{\enn}{\end{eqnarray*}}
\newcommand{\ba}{\backslash}
\newcommand{\pa}{\partial}

\newcommand{\ov}{\overline}
\newcommand{\curl}{{\rm curl\,}}

\newcommand{\eps}{\epsilon}

\newcommand{\om}{\omega}

\newcommand{\hx}{\hat{x}}
\newcommand{\bbS}{\mathbb{S}^2}
\newcommand{\bE}{\mathbf{E}}
\newcommand{\bH}{\mathbf{H}}

\title{Inverse electromagnetic source scattering problems with multi-frequency sparse phased and phaseless far field data}

\author{Xia Ji\thanks{LSEC, Academy of Mathematics and Systems Science, Chinese Academy of Sciences, Beijing 100190, China ({\tt jixia@lsec.cc.ac.cn}).}
\and
Xiaodong Liu\thanks{Institute of Applied Mathematics, Academy of Mathematics and Systems Science, Chinese Academy of Sciences, 100190 Beijing, China({\tt xdliu@amt.ac.cn}).}
}
\begin{document}

\maketitle

\begin{abstract}
This paper is concerned with uniqueness, phase retrieval and shape reconstruction methods for solving inverse electromagnetic source scattering problems with multi-frequency sparse phased or phaseless far field data.
With the phased data, we show that the smallest strip containing the source support with the observation direction as the normal can be uniquely determined by the multi-frequency far field pattern at a single observation direction. A phased direct sampling method is also proposed to reconstruct the strip.
The phaseless far field data is closely related to the outward energy flux, which can be measured more easily in practice. We show that the phaseless far field data is invariant under the translation of the sources, which implies that the location of the sources can not be uniquely recovered by the data.
To solve this problem, we consider simultaneously the scattering of magnetic dipoles with one fixed source point and at most three scattering strengths. With this technique, a fast and stable phase retrieval approach is proposed based on a simple geometrical result which provides a stable reconstruction of a point in the plane from three distances to the given points. A novel phaseless direct sampling method is also proposed to reconstruct the strip. The phase retrieval approach can also be combined with the phased direct sampling method to reconstruct the strip. Finally, to obtain a source support, we just need to superpose the indicators with respect to the sparse observation directions.
Extended numerical examples in three dimensions are conducted with noisy data, and the results further verify the effectiveness and robustness of
the proposed phase retrieval approach and direct sampling methods.

\vspace{.2in} {\bf Keywords:}
Electromagnetic source scattering, phaseless far field data, uniqueness, phase retrieval, direct sampling methods.

\vspace{.2in} {\bf AMS subject classifications:}
35R30, 35P25, 78A46

\end{abstract}
\section{Introduction}
In this paper, we consider the inverse source problem for the Maxwell system. Such a problem arises from many applications such as  antenna synthesis \cite{Balanis} and biomedical imaging \cite{ABF}. A specific example is the magnetoencephalography (MEG) \cite{DFK,FKM} in which it is desired to determine the location of electric activity within the brain from the induced magnetic fields outside the head. Motivated by these significant applications, the inverse source problem has been exhaustively studied in the literature \cite{ACTV, AM, AMR, ABF, BN, BLT, BLT2, BC, Bojarski, DFK, DML, DW, EV, FKM, HKP, HR, KW, MDZ, MKB, MZ, PD, Valdivia, WMGL, WSGLL}.
The primary difficulty is that the inverse source problem at a fixed frequency does not have a unique solution due to non-radiating sources \cite{AM,BC}. Possible solutions can be derived using models involving a priori knowledge of the sources or using multi-frequency measurements \cite{ABF,AM,MDZ,LU,Valdivia,WMGL,WSGLL}.
In particular, the multi-frequency measurements can be used to improve the stability of the inverse electromagnetic source problems \cite{BLZ}. 

In many cases of practical interest, it is difficult to implement measurement of a complex-valued electric field due to oscillations of its argument, so one has to recover the sources from the phaseless (intensity only) data. We refer to \cite{CH} for a direct imaging method for extended obstacle reconstruction using  phaseless electromagnetic total field data. 
However, difficulties arise if the phaseless far field data is considered. Actually, we will show in the next section that phaseless electric far field pattern of the scattered electromagnetic wave is invariant under the translation of the source, thus the source cannot be uniquely determined by the phaseless far field data, even multi-frequency data are considered. The same problem also appears in the acoustic and elastic source scattering problems \cite{JL, JLZ-source}. Other than developing methods using  phaseless data directly, we take an alternative approach to first recover the phase information. As a consequence, methods using phased data can be employed.
Recently in \cite{JLZ-obstacle}, by adding a point-like obstacle into the scattering system, a fast and stable phase retrieval approach is proposed for the inverse acoustic obstacle scattering problems. Such a phase retrieval approach is further modified by adding point sources with a fixed source point and at most three different scattering strengths in \cite{JL, JLZ-obstacle2, JLZ-source}. The phase retrieval approach in \cite{JL, JLZ-obstacle, JLZ-obstacle2, JLZ-source} is based on a simple geometrical result which provides a stable reconstruction of a point in the plane from three distances to the given points. 
A proper choice of the three given points is the key of the success of the stable phase retrieval approach.
We also refer to \cite{ZGLL} for a different phase retrieval technique with several specially arranged reference point sources.

The second case of tremendous practical interest is the sparse data.
Uniqueness results for determine a source usually assume that the measurements are taken at an infinite number of points around the source \cite{ABF,WMGL,WSGLL}.
For sparse data, in general, one cannot expect uniqueness for inverse problems. To determine the scattering objects, additional assumptions have to be considered. For example, in \cite{HLZ}, it is shown that a perfectly conducting ball can be determined by the electric far field patterns at three observation directions. Unfortunately, it seems that there are no corresponding results in the literature for the source determination.
From a practical perspective, it is reasonable and interesting to ask the question: {\em What kind of information can be extracted if the the measurements can only be taken at a fixed or a few (finitely many) points or directions.} Recently in \cite{HLLZ}, it is shown that the maximum and minimum distance between measurement point and the support of the source can be uniquely determined by using the electric field at a single point. The idea is from \cite{AHLS,GS,SK} for inverse acoustic source problem with multi-frequency sparse far field data. 

In this paper, we study the inverse electromagnetic source scattering problems using the multi-frequency phased or phaseless electric far field patterns taken at a fixed or a few observation directions. To our best knowledge, this is the first work on uniqueness, direct sampling methods and phase retrieval scheme for inverse electromagnetic scattering problems using phased or phaseless far field data at a single observation direction. We show that under certain condition the strip containing the source support with the observation direction as the normal can be uniquely determined by the multi-frequency far field patterns at a single observation direction. A quite simple direct sampling method (DSM) is then proposed to reconstruct the strip.  The phaseless far field data is also quite important due to its close connection to the well known Poynting vector and the flux density of energy transport through a closed surface. However, difficulties arise because of the translation invariant property of the phaseless far field data. To overcome the non-uniqueness difficulty, we introduce a magnetic dipole into the scattering system. Such a technique enables us to establish a uniqueness result and to provide a novel DSM for the source support reconstruction. Furthermore, with properly chosen scattering strength and polarization, a simple and stable phase retrieval method is proposed.

This work is a nontrivial extension of the results in \cite{JLZ-obstacle2, JLZ-source} for the inverse acoustic scattering problem of the Helmholtz equation to the inverse electromagnetic scattering of the Maxwell equations.
The Maxwell equations are more challenging because of the complexity of the vector fields. In particular, some sophisticated choices on polarization directions are required.

This paper is organized as follows.
In the next section, we formulate the direct and inverse electromagnetic source scattering problem and study the properties of the phaseless electric far field pattern. In section \ref{Sec-phased}, we study what kind of information of the source can be determined by the multi-frequency phased electric far field pattern at a fixed observation direction. We proceed to study the phaseless inverse problem in section \ref{Sec-phaseless}. Finally, in section \ref{Sec-Num}, some numerical simulations are presented to validate the effectiveness and robustness of the proposed DSM and phase retrieval algorithm.

\section{Electromagnetic source scattering problems}\label{ESSP}
This section is devoted to address the time harmonic electromagnetic source scattering problems.
We begin with the notations used throughout this paper. Vectors are distinguished from scalars by the use of bold typeface.
For a vector $\mathbf{a}:=(a_1, a_2, a_3)^{\rm T}\in\C^3$, where the superscript $``{\rm T}"$ denotes the transpose, we define the Euclidean norm of $\mathbf{a}$ by $|\mathbf{a}|:=\sqrt{\mathbf{a\cdot \ov{a}}}$, where
$\ov{\mathbf{a}}:=(\ov{a_1}, \ov{a_2}, \ov{a_3})^{\rm T}\in\C^3$ and $\ov{a_j}$ is the complex conjugate of $a_j$.
Denote by $\bbS:=\{\mathbf{x}\in \R^3: |\mathbf{x}|=1\}$ the unit sphere in $\R^3$.

Consider electromagnetic wave propagation in an isotropic medium in $\R^3$ with space independent electric permittivity $\eps$, magnetic permeability $\mu$ and
electric conductivity $\sigma$. We assume that the conductivity $\sigma\equiv 0$ in $\R^3$. Then the time harmonic electromagnetic wave with the angular frequency $\om$ is described by the electric field ${\bf E}$ and the magnetic field ${\bf H}$ satisfying the {\em Maxwell equations} \cite{Monk}
\be\label{Maxwell}
\curl {\bf E} -i\om\mu {\bf H} = 0, \quad \curl {\bf H} + i\om\eps {\bf E} = {\bf J}\quad \mbox{in}\,\, \R^3,
\en
where $i=\sqrt{-1}$ and ${\bf J}$ describes the external sources of electromagnetic disturbances.
For simplicity, we always assume that ${\bf J=J(y)}\in (L^\infty(\R^3))^3$ with compact support in $D$, where $D\subset\R^3$ is a bounded Lipschitz domain in $\R^3$
with connected complement.
Furthermore, the scattered fields ${\bf E}$ and ${\bf H}$ have to satisfy the
{\em Silver-M$\ddot{u}$ller radiation condition}
\be\label{SMRC}
\sqrt{\frac{\mu}{\eps}}{\bf H\times x- |x|E}= O\left(\frac{1}{\bf |x|}\right) \quad\mbox{as}\,\, {\bf |x|}\rightarrow\infty,
\en
uniformly w.r.t. $\hat{\mathbf{x}}:=\mathbf{x}/|\mathbf{x}|\in \bbS$.

Eliminating the magnetic field ${\bf H}$ from the Maxwell equations \eqref{Maxwell} leads to
\be\label{E}
\curl\curl {\bf E} -k^2 {\bf E} = i\om\mu {\bf J}\quad \mbox{in}\,\, \R^3,
\en
where $k:=\om\sqrt{\eps\mu}$ denotes the wavenumber. Fixing two wave numbers $0<k_{min}<k_{max}$, we consider the wave equation \eqref{E} with
\be\label{kassumption}
k\in K:=(k_{min}, k_{max}).
\en
It is well known (see e.g. \cite{CK,Monk}) that every radiating solution of \eqref{E} has an asymptotic behavior of the form
\be\label{Easy}
{\bf E}({\bf x},k; \mathbf{J})=\frac{e^{ik|{\bf x}|}}{4\pi|{\bf x}|}{\bf E}^{\infty}({\bf\hx},k;\mathbf{J}) + O\left(\frac{1}{|{\bf x}|^2}\right),\quad |{\bf x}|\rightarrow\infty,
\en
uniformly w.r.t. ${\bf\hx}$. The vector field ${\bf E}^{\infty}({\bf\hx},k;\mathbf{J})$ is known as the electric far field pattern of ${\bf E}$. It is an analytic function on the unit sphere $\bbS$ with respect to ${\bf \hx}$ and is a tangential field, i.e., ${\bf \hx\cdot E^{\infty}}({\bf \hx},k;\mathbf{J})=0$ for all ${\bf\hx}\in\bbS$. In this paper, the first question of particular interest is as follows.\\
\begin{itemize}
  \item {\bf IP1:} {\em What kind of information of the source $\mathbf{J}$ can be determined by the electric far field $\bE^{\infty}(\mathbf{\hx}_0, k; \mathbf{J}), k\in K$ at a single observation direction ${\bf \hx}_0$}?\\
\end{itemize}

To compute the electric far field pattern $\bE^{\infty}$ of the scattered field, we consider the scalar data $\mathbf{e}\cdot \bE^{\infty}$ for some direction $\mathbf{e}\in\bbS$. Clearly, taking $\mathbf{e}$ to be each of the orthogonal unit vectors $\mathbf{e_1}:=(1,0,0)^{\rm T},\, \mathbf{e_2}:=(0,1,0)^{\rm T}$ and $\mathbf{e_3}:=(0,0,1)^{\rm T}$, we can compute the three components of $\bE^{\infty}$. In fact since $\mathbf{\hx\cdot E^{\infty}(\mathbf{\hx})}=0$, we need only to use two tangential unit vectors for each $\bf\hx$. For any $\bf\hx\in\bbS$, choose a vector $\mathbf{q}$ such that $\mathbf{\hx\times q}\neq 0$. Then we can define vectors $\mathbf{l}$ and $\mathbf{m}$ by
\be\label{lm}
\mathbf{l}:=\frac{\mathbf{\hx\times q}}{|\mathbf{\hx\times q}|} \quad\mbox{and}\quad  \mathbf{m}:= \mathbf{\hx}\times \mathbf{l}.
\en
The three-tuple $(\mathbf{\hx, l, m})$ forms an orthonormal coordinate system in $\R^3$. Then $\mathbf{l}$ and $\mathbf{m}$ are two desired tangential unit vectors,
and
\be\label{Einflm}
\bE^{\infty}(\mathbf{\hx})= (\mathbf{l}\cdot \bE^{\infty}(\mathbf{\hx}) ) \mathbf{l} + (\mathbf{m}\cdot \bE^{\infty}(\mathbf{\hx}) ) \mathbf{m}.
\en

We recall the {\em dyadic Green's function} for Maxwell equations given by
\ben
\mathbb{G}_k(\mathbf{x,y}):=\Phi_k(\mathbf{x,y})\mathbb{I}+\frac{1}{k^2}\nabla_{\bf y}\nabla_{\bf y}\Phi_k(\mathbf{x,y}),\quad \mathbf{x\neq y},
\enn
where $\mathbb{I}$ is the $3\times3$ identity matrix and $\nabla_{\bf y}\nabla_{\bf y}\Phi_k(\mathbf{x,y})$ is the Hessian matrix for $\Phi$ defined by
\ben
\left(\nabla_{\bf y}\nabla_{\bf y}\Phi_k(\mathbf{x,y})\right)_{m,n} = \frac{\pa ^2 \Phi_k}{\pa y_m \pa y_n},\quad 1\leq m,n\leq 3.
\enn
Here, $\Phi_k$ denotes the fundamental solution of the scalar Helmholtz equation in $\R^3$ given by
\ben
\Phi_k({\bf x,y}):=\frac{e^{ik|\mathbf{x-y}|}}{4\pi|\mathbf{x-y}|}, \quad \mathbf{x\neq y}.
\enn
Then the scattered field ${\bf E}$ has the representation
\be\label{Erepresentation}
{\bf E}({\bf x}, k; \mathbf{J})=i\om\mu\int_{\R^3}\mathbb{G}_k(\mathbf{x,y}){\bf J(y)}d{\bf y}, \quad \mathbf{x}\in\R^3.
\en
Straightforward calculations show that the corresponding electric far field pattern is given by
\be\label{Einf}
{\bf E}^{\infty}({\bf \hx}, k; \mathbf{J})=i\om\mu(\mathbb{I}-{\bf\hx\hx}^{\rm T})\int_{\R^3}e^{-ik{\bf\hx\cdot y}}{\bf J(y)}d{\bf y}, \quad \mathbf{\hx}\in\bbS.
\en

Physically, the quantity $\Re(\bE\times\ov{\bH})$ is the well known {\em Poynting vector} and
\ben
F:=\Re\int_{\pa D}\nu\cdot\bE\times\ov{\bH}ds
\enn
gives the flux density of energy transport through the surface $\pa D$ (see p.79 in \cite{Cessenat}).

\begin{theorem}\label{Fproperties}
Let $B_R$ be a ball centered at the origin with radius $R$ large enough such that $\ov{D}\subset B_R$. Then we have the conservation of energy result
\be\label{FR}
\Re\int_{\pa D}\nu\cdot\bE\times\ov{\bH}ds = \Re\int_{\pa B_R}\nu\cdot\bE\times\ov{\bH}ds
\en
and
\be\label{FEinf}
F = \frac{1}{16\pi^2}\sqrt{\frac{\eps}{\mu}}\int_{\bbS}|\bE^{\infty}|^2 ds.
\en
\end{theorem}
\begin{proof}
We apply Gauss's divergence theorem in the domain $B_R\ba\ov{D}$ and the Maxwell equations \eqref{Maxwell} to deduce that
\ben
&&\int_{\pa B_R}\nu\cdot\bE\times\ov{\bH}ds - \int_{\pa D}\nu\cdot\bE\times\ov{\bH}ds\cr
&=&\int_{B_R\ba\ov{D}}(\curl\bE\cdot\ov{\bH} -\bE\cdot\curl\ov{\bH})d\mathbf{x}\cr
&=&\int_{B_R\ba\ov{D}}(i\om\mu|\bH|^2 +i\om\eps|\bE|^2)d\mathbf{x}.
\enn
Then the result on conservation of energy \eqref{FR} follows by taking the real part on both sides of the above equation.

It is well known that the electric field $\bE$ satisfies the finiteness condition
\be
\bE({\bf x}) = O\left(\frac{1}{|\bf x|}\right), \quad \bf |x|\rightarrow\infty,
\en
uniformly for all directions. Note that $F$ is independent of the radius $R$.
Furthermore, letting $R\rightarrow\infty$, using the {\em Silver-M$\ddot{u}$ller radiation condition} \eqref{SMRC} and the asymptotic behavior \eqref{Easy}, we find that
\ben
F
&=& \Re\int_{\pa B_R}\nu\cdot\bE\times\ov{\bH}ds\cr
&=& \lim_{R\rightarrow\infty}\Re\int_{\pa B_R}\ov{\bH}\times\nu\cdot\bE ds\cr
&=& \sqrt{\frac{\eps}{\mu}}\lim_{R\rightarrow\infty}\int_{\pa B_R}[|\bE|^2+O(1/R^3)] ds\cr
&=& \frac{1}{16\pi^2}\sqrt{\frac{\eps}{\mu}}\int_{\bbS}|\bE^{\infty}|^2 ds
\enn
and the proof is finished.
\end{proof}

Theorem \ref{Fproperties} shows that the phaseless electric far field pattern $|\bE^{\infty}|$ is closely related to the flux density of energy. Particularly, for radar applications, the phaseless electric far field pattern $|\bE^{\infty}(\mathbf{\hx},k;\mathbf{J})|$ defines the {\em radar cross section} in the direction $\mathbf{\hx}\in\bbS$ (see p. 392 in \cite{Monk}). Inspired by these applications, recall the two vectors $\mathbf{l}$ and $\mathbf{m}$ given in \eqref{lm}, we are also interest in the following question.\\
\begin{itemize}
  \item {\bf IP2:} {\em What kind of information of the source ${\bf J}$ can be determined by the phaseless electric far field pattern $|\mathbf{e}\cdot\bE^{\infty}(\mathbf{\hx}_0,k;\mathbf{J})|, k\in K, \mathbf{e}\in\{\mathbf{l},\mathbf{m}\}$  at a single observation direction ${\bf \hx}_0$}?\\
\end{itemize}

Unfortunately, the following theorem implies that the phaseless electric far field pattern is invariant under the translation of the source ${\bf J}$.

\begin{theorem}
Let $\mathbf{J}_{\mathbf{h}}(\mathbf{y}):=\mathbf{J}(\mathbf{y+h})$ be the shifted source with a fixed vector $\mathbf{h}\in \R^{3}$.
Then, for any fixed wavenumber $k>0$,
\be\label{Trans-Invariance}
\bE^{\infty}(\mathbf{\hx},k;\mathbf{J}_{\mathbf{h}}) = e^{ik{\bf\hx\cdot h}} \bE^{\infty}(\mathbf{\hx},k;\mathbf{J}), \quad \hat{\mathbf{x}}\in\bbS.
\en
\end{theorem}
\begin{proof}
Using the representation \eqref{Einf},
\ben
{\bf E}^{\infty}({\bf \hx}, k; \mathbf{J}_{\mathbf{h}})
&=&i\om\mu(\mathbb{I}-{\bf\hx\hx}^{\rm T})\int_{\R^3}e^{-ik{\bf\hx\cdot y}}{\bf J_{h}(y)}d{\bf y}\cr
&=&i\om\mu(\mathbb{I}-{\bf\hx\hx}^{\rm T})\int_{\R^3}e^{-ik{\bf\hx\cdot y}}{\bf J(y+h)}d{\bf y}\cr
&=&i\om\mu(\mathbb{I}-{\bf\hx\hx}^{\rm T})\int_{\R^3}e^{-ik{\bf\hx\cdot (z-h)}}{\bf J(z)}d{\bf z}\cr
&=&i\om\mu(\mathbb{I}-{\bf\hx\hx}^{\rm T})e^{ik{\bf\hx\cdot h}}\int_{\R^3}e^{-ik{\bf\hx\cdot z}}{\bf J(z)}d{\bf z}\cr
&=&e^{ik{\bf\hx\cdot h}} \bE^{\infty}(\mathbf{\hx},k;\mathbf{J}), \quad \mathbf{\hx}\in\bbS.
\enn
\end{proof}

From \eqref{Trans-Invariance}, it is clear that
\be\label{Translationinvariance}
|\bE^{\infty}(\mathbf{\hx},k;\mathbf{J}_{\mathbf{h}})| = | \bE^{\infty}(\mathbf{\hx},k;\mathbf{J})|,\quad \hat{\mathbf{x}}\in\bbS.
\en
That is, the location of the source can not be uniquely determined by the phaseless electric far field pattern $| \bE^{\infty}(\mathbf{\hx},k;\mathbf{J})|$, even multi-frequency is considered. Actually, even the location is known in advance, the source ${\bf J}$ can not be uniquely determined by noting the fact that
$|\bE^{\infty}(\mathbf{\hx},k;c\mathbf{J})| = | \bE^{\infty}(\mathbf{\hx},k;\mathbf{J})|, \hat{\mathbf{x}}\in\bbS$ for any constant $c\in\C$ with modulus one.

To overcome the above difficulties, we consider a magnetic dipole located at ${\bf z_0}\in\R^3$ with scattering strength $\tau\in\C$, polarization ${\bf p}\in\bbS$. The corresponding scattered electric field takes the form \cite{CK}
\be\label{Ez}
\bE_{\bf z_0}({\bf x},k; \tau, \mathbf{p}) = \tau \curl_{\bf x} \mathbf{p} \Phi({\bf x, z_0}), \quad {\bf x}\in\R^3\ba\{\bf z_0\}.
\en
The electric far field pattern corresponding to the electric field $\bE_{\bf z_0}$ is given by
\be\label{Ezinf}
\bE^{\infty}_{\bf z_0}({\bf \hx},k; \tau, \mathbf{p}) = ik\tau e^{-ik{\bf\hx\cdot z_0}}{\bf\hx}\times\mathbf{p},  \quad \mathbf{\hx}\in\bbS.
\en

Due to the translation invariance \eqref{Translationinvariance}, we add a magnetic dipole into the scattering system. Because of linearity, the resulting electric far field $\bE^{\infty}(\mathbf{\hx},k;\mathbf{J},\tau)$ is given by the sum of the electric far field pattern $\bE^{\infty}(\mathbf{\hx},k;\mathbf{J})$ due to the current source $\mathbf{J}$ and the electric far field pattern $\bE^{\infty}_{\bf z_0}({\bf \hx},k;\tau)$ due to the known dipole, i.e.,
\be\label{Einfsum}
\bE^{\infty}_{\bf z_0}(\mathbf{\hx},k;\mathbf{J},\tau, \mathbf{p}) = \bE^{\infty}(\mathbf{\hx},k;\mathbf{J}) + \bE^{\infty}_{\bf z_0}({\bf \hx},k;\tau, \mathbf{p}), \quad \mathbf{\hx}\in\bbS,
\en
We want to remark that the special case $\tau=0$ implies that no additional dipole exists.
Instead of studying the inverse problem {\bf IP2}, we turn to consider the following problem with updated phaseless data.\\
\begin{itemize}
  \item {\bf IP3:} {\em What kind of information of the source ${\bf J}$ can be determined by the phaseless electric far field pattern $|\mathbf{e}\cdot\bE^{\infty}_{\bf z_0}(\mathbf{\hx}_0,k;\mathbf{J},\tau, \mathbf{p})|, k\in K, \mathbf{e}\in\{\mathbf{l},\mathbf{m}\}$  at a single observation direction ${\bf \hx}_0$}?\\
\end{itemize}

In the next sections, we consider the uniqueness and sampling methods for {\bf IP1} with the phased data and for {\bf IP3} with the phaseless data. By properly choosing the dipole strengths, a simple phase retrieval scheme will also be proposed.

\section{Uniqueness and DSM for {\bf IP1}}\label{Sec-phased}
In this section, we investigate what kind of information of the source can be uniquely determined by the multi-frequency phased electric far field pattern at a single observation direction $\mathbf{\hx_0} \in \bbS$.
Denote by $D$ the support of the source $\mathbf{J}$, the ${\bf\hx}$-strip hull of $D$ for an observation direction $\mathbf{\hx} \in \bbS$ is defined by
\ben
S(\mathbf{\hx; J}):=  \{ \mathbf{y}\in \mathbb \R^{3}\; | \; \inf_{\mathbf{z}\in D}\mathbf{z}\cdot \mathbf{\hx} \leq \mathbf{y}\cdot \mathbf{\hx} \leq \sup_{\mathbf{z}\in D}\mathbf{z}\cdot \mathbf{\hx}\},
\enn
which is the smallest strip (region between two parallel hyper-planes) with $\pm \mathbf{\hx}$ as normals that contains $\overline{D}$.
Let
\ben
\Pi_{\alpha}:=\{\mathbf{y}\in \R^3 | \mathbf{y}\cdot\mathbf{\hx}+\alpha=0\}, \quad \alpha\in\R,
\enn
be a hyperplane with normal $\mathbf{\hx}$. Define
\be
\hat{\bf J}_{\mathbf{e}}(\alpha):=\int_{\Pi_\alpha}\mathbf{e}\cdot\mathbf{J}(\mathbf{y})ds(\mathbf{y}), \quad \mathbf{e}\in\bbS, \alpha\in\R.
\en

We first give a uniqueness result following the ideas for the inverse acoustic and elastic scattering problems \cite{AHLS, JL}.

\begin{theorem}\label{uni-strip}
For any fixed $\mathbf{\hx_{0}}\in\bbS$, choose a vector $\mathbf{q_0}\in\mathbb{S}^2$ such that $\mathbf{q_0\cdot\hx_0}\neq 0$ and then define two vectors $\mathbf{l}$ and $\mathbf{m}$ as in \eqref{lm}. For any $\mathbf{e}\in\{\mathbf{l,m}\}$, if the set
\be\label{set1}
\{\alpha\in\R | \, \Pi_\alpha\subset S(\mathbf{\hx}_0; \mathbf{J}), \hat{\bf J}_{\mathbf{e}}(\alpha)=0\}
\en
has Lebesgue measure zero, then the strip $S(\mathbf{\hx}_0; \mathbf{J})$ of the source support $D$ can be uniquely determined by
the data $\mathbf{e}\cdot\bE^{\infty}(\mathbf{\hx_0}, k; \mathbf{J})$ for all $k\in K$ at the observation direction $\mathbf{\hx_{0}}\in\bbS$.
\end{theorem}
\begin{proof}
From the representation \eqref{Einf}, by noting the fact that $\mathbf{e}\cdot\mathbf{\hx} = 0$ for all $\mathbf{e}\in\{\mathbf{l,m}\}$, we deduce that
\ben
\mathbf{e}\cdot\bE^{\infty}(\mathbf{\hx_0}, k; \mathbf{J})
&=&i\om\mu\int_{\R^3}e^{-ik{\bf\hx\cdot y}}\mathbf{e}\cdot{\bf J(y)}d{\bf y}\cr
&=&i\om\mu\int_{\R}e^{ik\alpha}\int_{\Pi_{\alpha}}\mathbf{e}\cdot{\bf J(y)}ds(\mathbf{y})d\alpha\cr
&=&i\om\mu\int_{\R}e^{ik\alpha}\hat{\bf J}_{\mathbf{e}}(\alpha)d\alpha, \quad \mathbf{e}\in\{\mathbf{l,m}\}.
\enn
Clearly, the data  $\frac{1}{i\om\mu}\mathbf{e}\cdot\bE^{\infty}(\mathbf{\hx_0}, k; \mathbf{J})$ is just the inverse Fourier transform of $\hat{\bf J}_{\mathbf{e}}(\alpha)$. This implies that $\hat{\bf J}_{\mathbf{e}}(\alpha)$ is uniquely determined by $\mathbf{e}\cdot\bE^{\infty}(\mathbf{\hx_0}, k; \mathbf{J}), k\in K$.
Note that here we have used the fact that the data $\mathbf{e}\cdot\bE^{\infty}(\mathbf{\hx_0}, k; \mathbf{J})$ is analytic in $k$.
Under the assumption that the set in \eqref{set1} has Lebesgue measure zero, it is seen that
\ben
S(\mathbf{\hx}_0;\mathbf{J})=\ov{\bigcup_{\alpha\in\R}\{\Pi_{\alpha} |\, \hat{\bf J}_{\mathbf{e}}(\alpha)\neq 0\}}
\enn
which implies that the strip $S(\mathbf{\hx}_0;\mathbf{J})$ is uniquely determined by $\hat{\bf J}_{\mathbf{e}}$, and also by $\mathbf{e}\cdot\bE^{\infty}(\mathbf{\hx_0}, k; \mathbf{J})$ for all $k\in K$. The proof is complete.
\end{proof}

\begin{remark}
We add some remarks on Theorem \ref{uni-strip}.
\begin{itemize}
  \item Clearly, the vector $\mathbf{e}$ can be any unit vector satisfying $\mathbf{e\cdot\hx_0}=0$.
  \item The set in \eqref{set1} has Lebesgue measure zero if the real part of a complex multiple of the source projection function $\mathbf{e\cdot J}$  is bounded away from zero on their support, i.e., we assume that $\mathbf{e\cdot J}\in L^{\infty}(D)$ satisfies
\be\label{fassumption}
\Re(e^{i\alpha}\mathbf{e\cdot J(y)})\geq c_{0},\quad a.e.\,\, \mathbf{y}\in D
\en
for some $\alpha \in \R$ and $c_{0} >0$.
However, Theorem \ref{uni-strip} is not true in general if the set in \eqref{set1} has positive Lebesgue measure.
For example, for $\mathbf{y}=(y_1,y_2,y_3)^{\rm T}\in\R^3$, we consider
\be\label{J1}
  \mathbf{J}_1(\mathbf{y})=\left\{
         \begin{array}{ll}
           (1,0,0)^{\rm T}, & \hbox{$y_1\in (-1,1),\, y_2\in (-2,-1]\cup[1,2),\,  y_3\in (-1,1)$;} \quad\\
           (y_1, 0,0)^{\rm T}, & \hbox{$y_1\in (-1,1),\,  y_2\in (-1,1),\,  y_3\in (-1,1)$;}\\
           (0,0,0)^{\rm T}, & \hbox{otherwise.}
         \end{array}
       \right.
\en
and
\be\label{J2}
  \mathbf{J}_2(\mathbf{y})=\left\{
         \begin{array}{ll}
           (1,0,0)^{\rm T}, & \hbox{$y_1\in (-1,1),\,  y_2\in (-2,-1]\cup[1,2),\,  y_3\in (-1,1)$;}\quad\\
           (0,0,0)^{\rm T}, & \hbox{otherwise.}
         \end{array}
       \right.
\en
Taking $\mathbf{\hx}_0 = (0,1,0)^{\rm T}$, then the strips with normals $\pm\mathbf{\hx}_0$ corresponding to $\mathbf{J}_1$ and $\mathbf{J}_2$ are given by
\ben
  S(\mathbf{\hx}_0;\mathbf{J}_1):=\R\times [-2,2]\times\R \quad\mbox{and}\quad S(\mathbf{\hx}_0;\mathbf{J}_2):=\R\times [-2,-1] \cup [1,2]\times\R,
\enn
respectively.
However, for any $\mathbf{e}\in\bbS$ satisfying $\mathbf{e\cdot\hx_0}=0$, straightforward calculations show that
\ben
  \mathbf{e}\cdot \bE^{\infty}(\mathbf{\hx_0}, k; \mathbf{J}_1) = \mathbf{e}\cdot \bE^{\infty}(\mathbf{\hx_0}, k; \mathbf{J}_2), \quad k\in K.
\enn
\end{itemize}
\end{remark}

Now we turn to the DSM on how to determine the strip $S(\mathbf{\hx_0};\mathbf{J})$ from the projection of the electric far field pattern in the direction $\mathbf{e} \in \{\mathbf{l,m}\}$ at a single observation direction $\mathbf{\hx_0}\in\bbS$. We consider the following integral
\be\label{G}
G_{\mathbf{\hx_0}}(\mathbf{z,e}) :=\int_{K}\mathbf{e}\cdot \bE^{\infty}(\mathbf{\hx_0}, k; \mathbf{J}) e^{ik\mathbf{\hx_0\cdot z}}dk, \quad \mathbf{z}\in\R^3, \quad \mathbf{e} \in \{\mathbf{l,m}\}
\en
and define indicator the
\be\label{indicator1}
I_{\mathbf{\hx_0}}(\mathbf{z,e}):=\left|G_{\mathbf{\hx_0}}(\mathbf{z,e})\right|, \quad \mathbf{z}\in\R^3, \quad \mathbf{e} \in \{\mathbf{l,m}\}.
\en
To see how this indicator might work, we first define $\Pi:=\{\mathbf{d}\in\R^3 | \mathbf{d}=a\mathbf{l}+b\mathbf{m}, a, b\in \R\}$. Then $\Pi$ is a hyper plane with normals $\pm\mathbf{\hx_0}$, that is, for any $\mathbf{d}\in\Pi$, we have $\mathbf{d\cdot\hx_0}=0$. Straight calculations show that
\ben
I_{\mathbf{\hx_0}}(\mathbf{z}+\alpha\mathbf{d},\mathbf{e})= I_{\mathbf{\hx_0}}(\mathbf{z,e}),\quad \mathbf{z}\in\R^3, \quad \mathbf{e} \in \{\mathbf{l,m}\},\quad \alpha\in\R,
\enn
which implies that the indicator $I_{\mathbf{\hx_0}}$ takes the same value for sampling points moving in the hyperplane with normals $\pm\mathbf{\hx_0}$.
Inserting the representation \eqref{Einf} into \eqref{indicator1}, using the formula $k=\om\sqrt{\eps\mu}$ and the fact that $\mathbf{e\cdot\hx_0}=0$, and then integrating by parts,  we deduce that
\be\label{Ibehaviorfar}
I_{\mathbf{\hx_0}}(\mathbf{z,e})
&=&\sqrt{\frac{\mu}{\eps}}\left|\int_{D}\mathbf{e}\cdot \mathbf{J(y)}\int_{K} k e^{ik\mathbf{\hx_0\cdot (z-y)}}dkd\mathbf{y}\right|\cr
&=&\sqrt{\frac{\mu}{\eps}}\left|\int_{D}\frac{f(\mathbf{y,z,e})}{|\mathbf{\hx_0\cdot (z-y)}|}d\mathbf{y}\right|, \quad \mathbf{z}\in\R^3, \quad \mathbf{e} \in \{\mathbf{l,m}\}
\en
with
\ben
f(\mathbf{y,z,e}):=\mathbf{e}\cdot \mathbf{J(y)}\left[k e^{ik\mathbf{\hx_0\cdot (z-y)}}|_{K}-\int_{K} e^{ik\mathbf{\hx_0\cdot (z-y)}}dk\right], \quad \mathbf{y}\in D, \,\,\mathbf{e} \in \{\mathbf{l,m}\}.
\enn
Clearly, $f(\mathbf{y,z,e})$ is uniformly bounded with respect to the sampling point $\mathbf{z}\in\R^3$, and thus the indicator $I_{\mathbf{\hx_0}}(\mathbf{z,e})$
decays like $\frac{1}{|\mathbf{\hx_0\cdot (z-y)}|}$ when the sampling point $\mathbf{z}$ moves away from the strip $S(\mathbf{\hx_0};\mathbf{J})$.

From the behaviors explained above, we expect that the indicator $I_{\mathbf{\hx_0}}$ will give a rough reconstruction for the strip $S(\mathbf{\hx_0};\mathbf{J})$.
We can now propose the DSM for strip reconstruction using broadband data at a single observation direction.\\

{\bf Strip Reconstruction Scheme One: }
\begin{itemize}{\em
  \item (1). Collect the broadband electric far field pattern $\mathbf{e}\cdot \bE^{\infty}(\mathbf{\hx_0}, k; \mathbf{J}), \, k\in K$ with some $\mathbf{e}\in \{\mathbf{l,m}\}$ where $\mathbf{l}$ and $\mathbf{m}$ are given in \eqref{lm}.
  \item (2). Select a hyperplane $\Sigma$ parallel to the observation direction $\mathbf{\hx_0}$. Select a sampling region in $\Sigma$ with a fine mesh $\mathcal {Z}$ containing the projection of the source support $D$ in the hyperplane $\Sigma$,
  \item (3). Compute the indicator functional $I_{\mathbf{\hx_0}}$  for all sampling points $z\in\mathcal {Z}$,
  \item (4). Plot the indicator functional $I_{\mathbf{\hx_0}}$ }.
\end{itemize}\,\quad

The sampling region in above scheme is only a bounded two dimensional domain in the hyperplane $\Sigma$, which is parallel to the observation direction $\mathbf{\hx_0}$. Thus, different to the other numerical methods, our numerical implementation is very fast.
If the source support $D$ has only one component, by plotting the indicator $I_{\mathbf{\hx_0}}$, two parallel lines with normal $\mathbf{\hx_0}$ are expected to be roughly reconstructed. Then the strip $S(\mathbf{\hx_0};\mathbf{J})$ is the region between two parallel hyperplane with normal $\pm\mathbf{\hx_0}$ containing the two reconstructed lines, respectively.

Recall that $\bE^{\infty}(\mathbf{\hx}_0)= (\mathbf{l}\cdot \bE^{\infty}(\mathbf{\hx}_0) ) \mathbf{l} + (\mathbf{m}\cdot \bE^{\infty}(\mathbf{\hx}_0) ) \mathbf{m}$. However, to determine the strip $S(\mathbf{\hx_0};\mathbf{J})$, we have used only the data $\mathbf{l}\cdot \bE^{\infty}(\mathbf{\hx}_0)$ or $\mathbf{m}\cdot \bE^{\infty}(\mathbf{\hx}_0)$, which is partial information of the phased electric far field pattern.
Theorem \ref{uni-strip} and the DSM also implies that three linearly independent observation directions are enough to give a rough support of the source function $\mathbf{J}$.

\section{DSM and phase retrieval method for {\bf IP3}}\label{Sec-phaseless}
We begin with a novel DSM for strip $S(\mathbf{\hx_0};\mathbf{J})$ reconstruction using phaseless data directly.
For any fixed $\mathbf{\hx_{0}}\in\bbS$, choose a vector $\mathbf{q_0}\in\mathbb{S}^2$ such that $\mathbf{q_0\cdot\hx_0}\neq 0$ and then define two vectors $\mathbf{l}$ and $\mathbf{m}$ as in \eqref{lm}. Taking $\tau_1\in\C\ba\{0\}$ and $\mathbf{z_0}\in\R^3\ba\ov{D}$, for the novel DSM, we need the following phaseless data
\be\label{DSMphaselessdata}
|\mathbf{m}\cdot\bE^{\infty}_{\bf z_0}(\mathbf{\hx}_0,k;\mathbf{J},\tau,\mathbf{l})|,\quad k\in K,\,\,\tau\in \{0, \tau_1\}.
\en
Using the representation \eqref{Einfsum}, noting that the three-tuple $\mathbf{\{\hx_0, l, m\}}$ forms an orthonormal basis of $\R^3$, we find
\be\label{H}
&&\mathcal {H}(\mathbf{\hx_0}, k; \mathbf{J},\tau,\mathbf{z_0,l,m})\cr
&:=& \frac{1}{k}\left[|\mathbf{m}\cdot\bE^{\infty}_{\bf z_0}(\mathbf{\hx}_0,k;\mathbf{J},\tau_1,\mathbf{l})|^2-|\mathbf{m}\cdot\bE^{\infty}(\mathbf{\hx}_0,k;\mathbf{J})|^2-|k\tau_1|^2\right] \cr
&=& \frac{1}{k}\left[|\mathbf{m}\cdot\bE^{\infty}(\mathbf{\hx}_0,k;\mathbf{J}) +ik\tau_1 e^{-ik\mathbf{\hx_0\cdot z_0}}|^2-|\mathbf{m}\cdot\bE^{\infty}(\mathbf{\hx}_0,k;\mathbf{J})|^2-|k\tau_1|^2\right] \cr
&=& -2\Re[\mathbf{m}\cdot\bE^{\infty}(\mathbf{\hx}_0,k;\mathbf{J})i\tau_1e^{ik\mathbf{\hx_0\cdot z_0}}]
\en
Furthermore, we introduce the following indicator
\be\label{Indicatorphaseless}
I_{\bf \hx_0,z_0}{\bf(z)}:=\left|\int_{K} \mathcal {H}(\mathbf{\hx_0}, k; \mathbf{J},\tau,\mathbf{z_0,l,m})\cos[k\mathbf{\hx_0\cdot(z-z_0)}]dk\right|,\quad\,\mathbf{z}\in\R^3.
\en
We observe immediately two obvious properties of the indicator $I_{\bf \hx_0,z_0}$:
\begin{itemize}
  \item $I_{\bf \hx_0,z_0}{\bf(z+d)}=I_{\bf \hx_0,z_0}{\bf(z)}$ for any $\mathbf{z}\in\R^3$ provided that $\mathbf{d\cdot\hx_0}=0$;
  \item $I_{\bf \hx_0,z_0}{\bf(2z_0-z)}=I_{\bf \hx_0,z_0}{\bf(z)}$ for any $\mathbf{z}\in\R^3$.
\end{itemize}
Difficulty will arise due to the second property. Actually, we hope to determine the strip $S(\mathbf{\hx_0, J})$ by using the novel indicator $I_{\bf \hx_0,z_0}$.
However, from the second property, the symmetric strip $S(\mathbf{\hx_0, J, z_0})$ of the strip $S(\mathbf{\hx_0, J})$ with respect to $\mathbf{z_0}$ will also be reconstructed. We can not distinguish these two strips. In particular, if $\mathbf{z_0}\in S(\mathbf{\hx_0, J})$, the indicator will fail to reconstruct the strip $S(\mathbf{\hx_0, J})$. The good news is that we have the freedom to choose the dipole position $\mathbf{z_0}$. Thus, by letting $\mathbf{z_0}$ move along the direction $\mathbf{\hx_0}$, we can always assume that $\mathbf{z_0}\notin S(\mathbf{\hx_0, J})$. To pick the correct strip, we introduce two techniques. The first one is to take a different dipole position $\mathbf{z_1}$ along the direction $\mathbf{\hx_0}$, and then the symmetric strip will change. The second one is to take the dipole position far away from the domain of interest.

We now explain why the indicator $I_{\bf \hx_0,z_0}$ can be used to reconstruct the strip $S(\mathbf{\hx_0, J})$.
Inserting \eqref{H} into \eqref{Indicatorphaseless}, straightforward calculations show that
\ben
I_{\bf \hx_0,z_0}{\bf(z)}=\left| \ov{\tau_1}G_{\mathbf{\hx_0}}(\mathbf{z,m})+\tau_1\ov{G_{\mathbf{\hx_0}}(\mathbf{z,m})}+\ov{\tau_1}G_{\mathbf{\hx_0}}(\mathbf{2z_0-z,m})+\tau_1\ov{G_{\mathbf{\hx_0}}(\mathbf{2z_0-z,m})}\right|
\enn
with the auxiliary $G_{\mathbf{\hx_0}}$ defined by \eqref{G}. Note that we have the freedom to choose the dipole position $\mathbf{z_0}$. Taking $\mathbf{z_0}:=r\mathbf{\hx_0}$ and letting $r\rightarrow\infty$, using the well known Riemann-Lebesgue Lemma, we find that
\ben
G_{\mathbf{\hx_0}}(\mathbf{2z_0-z,m})\rightarrow 0, \quad\mbox{as}\quad |\mathbf{z_0}|=r\rightarrow\infty.
\enn
Thus, we have
\ben
I_{\bf \hx_0,z_0}{\bf(z)}\approx\left|\ov{\tau_1}G_{\mathbf{\hx_0}}(\mathbf{z,m})+\tau_1\ov{G_{\mathbf{\hx_0}}(\mathbf{z,m})}\right|, \quad\mbox{as}\quad |\mathbf{z_0}|=r\rightarrow\infty.
\enn
Motivated by this, we expect that the novel indicator $I_{\bf \hx_0,z_0}{\bf(z)}$ behaves like the indicator $I_{\mathbf{\hx_0}}{\bf(z,e)}$ with phased data, and thus can be used to reconstruct the strip $S(\mathbf{\hx_0, J})$.\\

{\bf Strip Reconstruction Scheme Two: }
\begin{itemize}{\em
  \item (1). Collect the broadband electric far field pattern $|\mathbf{m}\cdot\bE^{\infty}_{\bf z_0}(\mathbf{\hx}_0,k;\mathbf{J},\tau,\mathbf{l})|,\quad k\in K,\,\,\tau\in \{0, \tau_1\}$ where $\mathbf{l}$ and $\mathbf{m}$ are given in \eqref{lm}.
  \item (2). Select a hyperplane $\Sigma$ parallel to the observation direction $\mathbf{\hx_0}$. Select a sampling region in $\Sigma$ with a fine mesh $\mathcal {Z}$ containing the projection of the source support $D$ in the hyperplane $\Sigma$,
  \item (3). Compute the indicator functional $I_{\bf \hx_0,z_0}{\bf(z)}$ for all sampling points $z\in\mathcal {Z}$,
  \item (4). Plot the indicator functional $I_{\bf \hx_0,z_0}{\bf(z)}$ }.
\end{itemize}\,\quad

The second way to deal with the phaseless problem {\bf IP3} is first reconstruct the phased data $\mathbf{m}\cdot\bE^{\infty}(\mathbf{\hx_0}, k; \mathbf{J})$ for all $k\in K$, and then using the known methods and results with phased data.
To do so, we recall a geometrical result on determining a point in the plane from the distances to three given points \cite{JLZ-obstacle}.
\begin{lemma}\label{phaseretrieval}
Let $z_j:=x_j+iy_j,\,j=1,2,3,$ be three different complex numbers such that they are not collinear.
Then there is at most one complex number $z\in\C$ with the distances $r_j=|z-z_j|,\,j=1,2,3$. Let further $\eps>0$ and assume that
\ben
|r_j^{\eps}-r_j|\leq\eps, \quad j=1,2,3.
\enn
Here, and throughout the paper, we use the subscript $\eps$ to denote the polluted data.
Then there exists a constant $c>0$ depending on $z_j, j=1,2,3$, such that
\ben
|z^{\eps}-z|\leq c\eps.
\enn
\end{lemma}

Define
\ben
\mathcal {T}:= \{\tau_1,\tau_2,\tau_3\},
\enn
where $\tau_1, \tau_2, \tau_3\in\C$ are three different scattering strengths such that $\tau_2-\tau_1$ and $\tau_3-\tau_1$ are linearly independent.
Recall that
\ben
|\mathbf{m}\cdot\bE^{\infty}_{\bf z_0}(\mathbf{\hx}_0,k;\mathbf{J},\tau_j,\mathbf{l})|
= |\mathbf{m}\cdot\bE^{\infty}(\mathbf{\hx}_0,k;\mathbf{J})+ik\tau_j e^{-ik\mathbf{z_0\cdot\hx_0}}|,\quad k\in K,\quad j=1,2,3.
\enn
For any fixed observation direction $\mathbf{\hx_0}\in \bbS$, wave number $k\in K$, we set
\be
\label{rj}r_j&:=&|\mathbf{m}\cdot\bE^{\infty}_{\bf z_0}(\mathbf{\hx}_0,k;\mathbf{J},\tau_j,\mathbf{l})|\\
\label{z} z&:=&\mathbf{m}\cdot\bE^{\infty}(\mathbf{\hx}_0,k;\mathbf{J}),
\en
and
\be\label{zj}
z_j:=-ik\tau_j e^{-ik\mathbf{z_0\cdot\hx_0}},\quad j=1,2,3.
\en

Combining Lemma \ref{phaseretrieval} and the uniqueness Theorem \ref{uni-strip} with phased data, we immediately derive the following uniqueness theorem with phaseless data.

\begin{theorem}\label{uni-strip-phaseless}
For any fixed $\mathbf{\hx_{0}}\in\bbS$, choose a vector $\mathbf{q_0}\in\mathbb{S}^2$ such that $\mathbf{q_0\cdot\hx_0}\neq 0$ and then define two vectors $\mathbf{l}$ and $\mathbf{m}$ as in \eqref{lm}. If the set
\be\label{set2}
\{\alpha\in\R | \, \Pi_\alpha\subset S(\mathbf{\hx}_0; \mathbf{J}), \hat{\bf J}_{\mathbf{m}}(\alpha)=0\}
\en
has Lebesgue measure zero, then the strip $S(\mathbf{\hx}_0; \mathbf{J})$ of the source support $D$ can be uniquely determined by
the phaseless data $|\mathbf{m}\cdot\bE^{\infty}_{\bf z_0}(\mathbf{\hx}_0,k;\mathbf{J},\tau,\mathbf{l})|$ for all $k\in K,\,\tau\in \mathcal {T}$ at the observation direction $\mathbf{\hx_{0}}\in\bbS$.
\end{theorem}

With the representations \eqref{rj} and \eqref{zj}, we use the following stable phase retrieval scheme \cite{JLZ-obstacle} to determine the phased data $\mathbf{m}\cdot\bE^{\infty}(\mathbf{\hx_0}, k; \mathbf{J})$ for all $k\in K$.\\

\begin{figure}[htbp]
\centering
\includegraphics[width=2.5in]{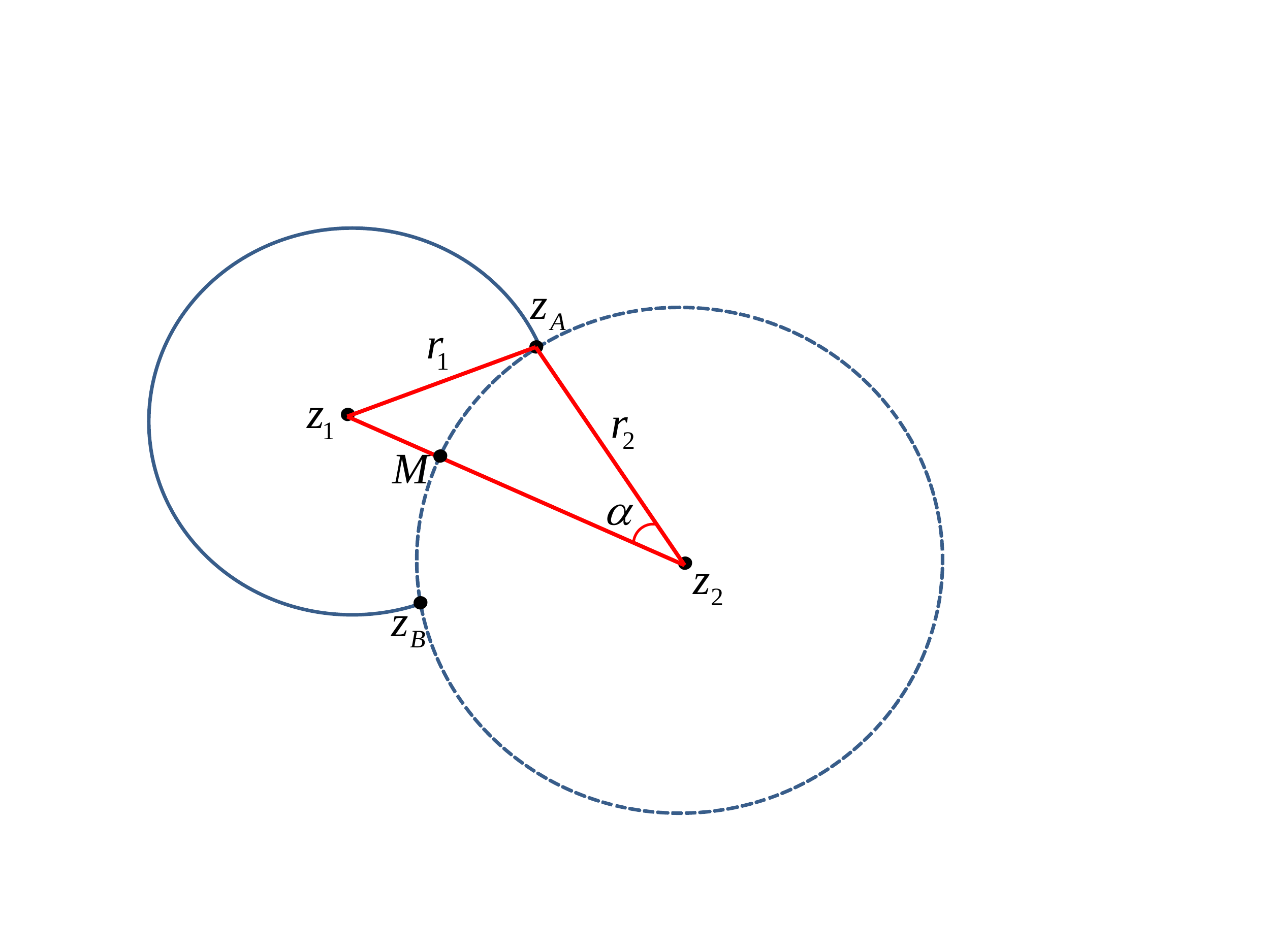}
\caption{Sketch map for phase retrieval scheme.}
\label{twodisks}
\end{figure}

{\bf Phase Retrieval Scheme \cite{JLZ-obstacle}:}

\begin{itemize}{\em
  \item (1). {\bf\rm Collect the distances $r_j:=|z-z_j|$ with given complex numbers $z_j,\,j=1,2,3$.} If $r_j=0$ for some $j\in\{1,2,3\}$, then $z=z_j$. Otherwise, go to next step.
  \item (2). {\bf\rm Look for the point $M=(x_M, y_M)$.} As shown in Figure \ref{twodisks}, $M$ is the intersection of circle centered at $Z_2$ with radius $r_2$ and the ray $z_2z_1$ with initial point $z_2$. Denote by $d_{1,2}:=|z_1-z_2|$ the distance between $z_1$ and $z_2$, then
       \be\label{xMyM}
       x_M=\frac{r_2}{d_{1,2}}x_1+\frac{d_{1,2}-r_2}{d_{1,2}}x_2,\quad y_M=\frac{r_2}{d_{1,2}}y_1+\frac{d_{1,2}-r_2}{d_{1,2}}y_2,
       \en
  \item (3). {\bf\rm Look for the points $z_A=(x_A, y_A)$ and $z_B=(x_B, y_B)$.} Note that $z_A$ and $z_B$ are just two rotations of $M$ around the point $z_2$. Let $\alpha\in [0,\pi]$ be the angle between rays $z_2z_1$ and $z_2z_A$. Then, by the law of cosines, we have
       \be\label{cosalpha}
       \cos \alpha=\frac{r_2^2+d_{1,2}^2- r_1^2}{2r_2d_{1,2}}.
       \en
       Note that $\alpha\in [0,\pi]$ and $\sin^2 \alpha+\cos^2 \alpha=1$, we deduce that $\sin \alpha=\sqrt{1-\cos^2 \alpha}$.
       Then
       \be
       \label{xA}x_A &=& x_2+\Re\{[(x_M-x_2)+i(y_M-y_2)]e^{-i\alpha}\},\\
       \label{yA}y_A &=& y_2+\Im\{[(x_M-x_2)+i(y_M-y_2)]e^{-i\alpha}\},\\
       \label{xB}x_B &=& x_2+\Re\{[(x_M-x_2)+i(y_M-y_2)]e^{i\alpha}\},\\
       \label{yB}y_B &=& y_2+\Im\{[(x_M-x_2)+i(y_M-y_2)]e^{i\alpha}\}.
       \en
  \item (4). {\bf\rm Determine the point $z$. $z=z_A$ if the distance $|z_Az_3|=r_3$, or else $z=z_B$.}}\\
\end{itemize}

Finally, we introduce the third strip reconstruction scheme using multi-frequency data at a single observation direction.\\

{\bf Strip Reconstruction Scheme Three: }
\begin{itemize}{\em
  \item (1). Collect the broadband electric far field pattern $|\mathbf{m}\cdot\bE^{\infty}_{\bf z_0}(\mathbf{\hx}_0,k;\mathbf{J},\tau,\mathbf{l})|,\quad k\in K,\,\,\tau\in \mathcal {T}$ where $\mathbf{l}$ and $\mathbf{m}$ are given in \eqref{lm}.
  \item (2). Compute the phased data $\mathbf{m}\cdot\bE^{\infty}(\mathbf{\hx_0}, k; \mathbf{J})$ for all $k\in K$ by using the {\rm\bf Phase Retrieval Scheme}.
  \item (3). Reconstruct the strip by using the {\rm\bf Strip Reconstruction Scheme One}.}
\end{itemize}\,\quad

\section{Numerical examples}\label{Sec-Num}

Now we present a variety of numerical examples in three dimensions to illustrate the applicability, effectiveness and robustness of our sampling methods with broadband sparse data.
The forward problems are computed  using equation \eqref{Einf}. We consider $ {\bf J}=(3/2,3\sqrt{3}/2,3/2)^{\rm T}$ and  five different supports of $ {\bf J}$.
\begin{itemize}
\item $ {\bf S_1}$: cubic $[0,1]^3$;
\item $ {\bf S_2}$: unit ball centered at the origin;
\item $ {\bf S_3}$: cubic $[0,1]^3$ and unit ball centered at $(3,3,0)^{\rm T}$;
\item $ {\bf S_4}$: L-shaped domain $[0,4]\times[0,4]\times [0,1]\backslash [1,4]\times[1,1]\times [0,1]$;
\item $ {\bf S_5}$: cuboid $[0,2]\times[0,1]\times [0,1]$.
\end{itemize}

For any fixed $\mathbf{\hx_{0}}\in\bbS$,  we assume to have multiple frequency phaseless far field
data $\bE^{\infty}(\mathbf{\hx}_0,k_j;\mathbf{J}),j=1,\cdots,M$, where $M=30,k_{min}=9.5,k_{max}=24$ such that  $k_j=(j-1)\times0.5+k_{min}$. We further perturb this data by relative error
\ben
|\mathbf{m}\cdot\bE^{\infty,\delta}_{\bf z_0}(\mathbf{\hx}_0,k_j;\mathbf{J},\tau,\mathbf{l})|= |\mathbf{m}\cdot\bE^{\infty}_{\bf z_0}(\mathbf{\hx}_0,k_j;\mathbf{J},\tau,\mathbf{l})|(1+\delta*\bbespilon), \quad j=1,2,\cdots, M,
\enn
where $\bbespilon$ is a uniformly distributed random number in the open interval $(-1,1)$. The value of $\delta$ is the error level.
We also consider absolute error in {\bf Example PhaseRetrieval}. In this case, we perturb the phaseless data in the form
\ben
|\mathbf{m}\cdot\bE^{\infty,\delta}_{\bf z_0}(\mathbf{\hx}_0,k_j;\mathbf{J},\tau,\mathbf{l})|= \max\Big\{0,\,|\mathbf{m}\cdot\bE^{\infty}_{\bf z_0}(\mathbf{\hx}_0,k_j;\mathbf{J},\tau,\mathbf{l})|+\delta*\bbespilon\Big\},\quad j=1,2,\cdots, M,
\enn
In the {\bf Strip Reconstruction Scheme One}, we perturb $\bE^{\infty}(\mathbf{\hx}_0,k_j;\mathbf{J}),j=1,\cdots,M$ directly in the same way.

In our numerical examples, we give the ${\bf x-y}$ plane projection and ${\bf y-z}$ plane projection of the reconstruction.  For the ${\bf x-y}$ plane projection, we choose the observation direction $\mathbf{\hx_{0}}$ as
$(\cos \theta_j, \sin \theta_j,0)^{\rm T}, \theta_j=(j-1)\pi/N, j=1,\cdots, N$. $N$ is the number of observation directions. For the ${\bf y-z}$ plane projection, we choose the observation direction $\mathbf{\hx_{0}}$ as
$(0, \cos \theta_j, \sin \theta_j)^{\rm T}, \theta_j=(j-1)\pi/N, j=1,\cdots, N$. $N=1,2,20$ are used.

In the simulations, we use $0.05$ as the sampling space. $\tau=0.1$ is used in {\bf Strip Reconstruction Scheme Two} and $\tau=\pm 0.1, 0.1i, {\bf z_0}=(2,2,0)^{\rm T}$ are used in {\bf Strip Reconstruction Scheme Three}. If not otherwise
stated, $10\%$ relative error is added and ${\bf x-y}$ plane projection is considered.

\subsection{{\bf Strip Reconstruction Scheme One} for cubic support (${\bf S_1}$)}
In this example, we consider the cubic support reconstruction with $1, 2$ and $20$ observation directions. Fig. \ref{square1} (a)
clearly shows that the source support lies in a strip, which is perpendicular to the observation direction. In Fig. \ref{square1} (b), since
the observation directions are perpendicular to each other, the strips are perpendicular to each
other too. The source support must be located in the cross sections of the two strips.  In Fig. \ref{square1} (c), with the increase of number of the observation directions, the location and size of the support are reconstructed correctly.\\

\begin{figure}[htbp]
  \centering
  \subfigure[\textbf{One observation direction.}]{
    \includegraphics[width=1.5in]{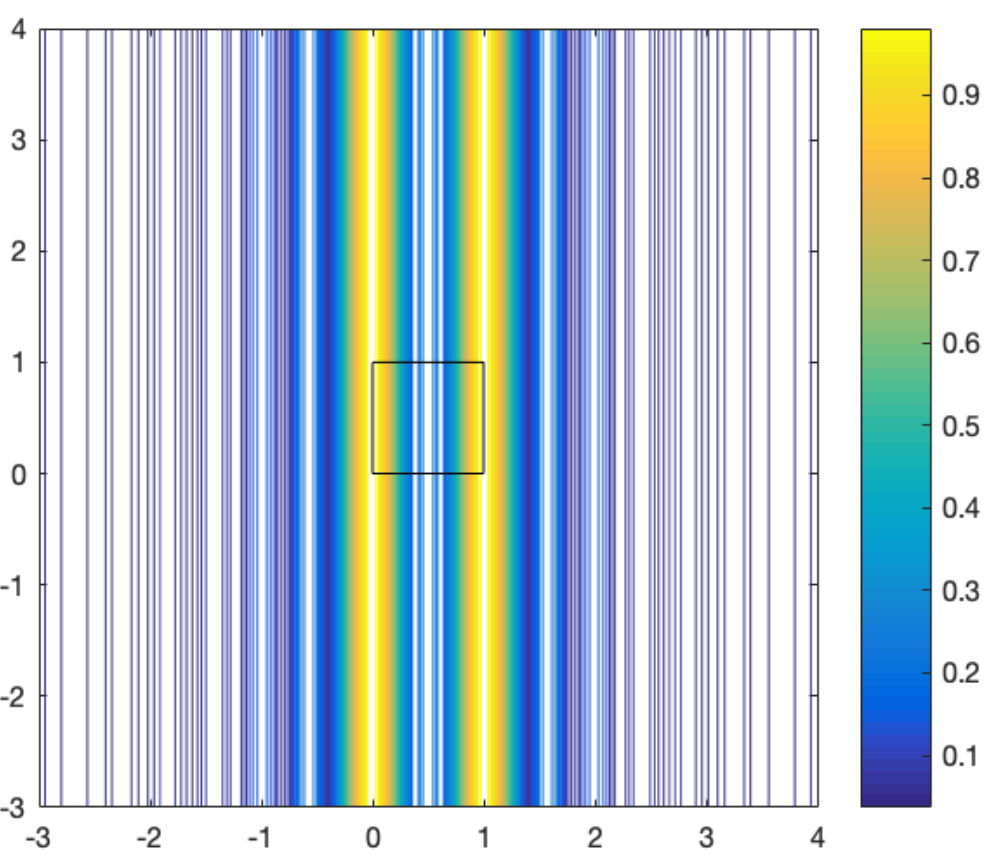}}
  \subfigure[\textbf{Two observation directions.}]{
    \includegraphics[width=1.5in]{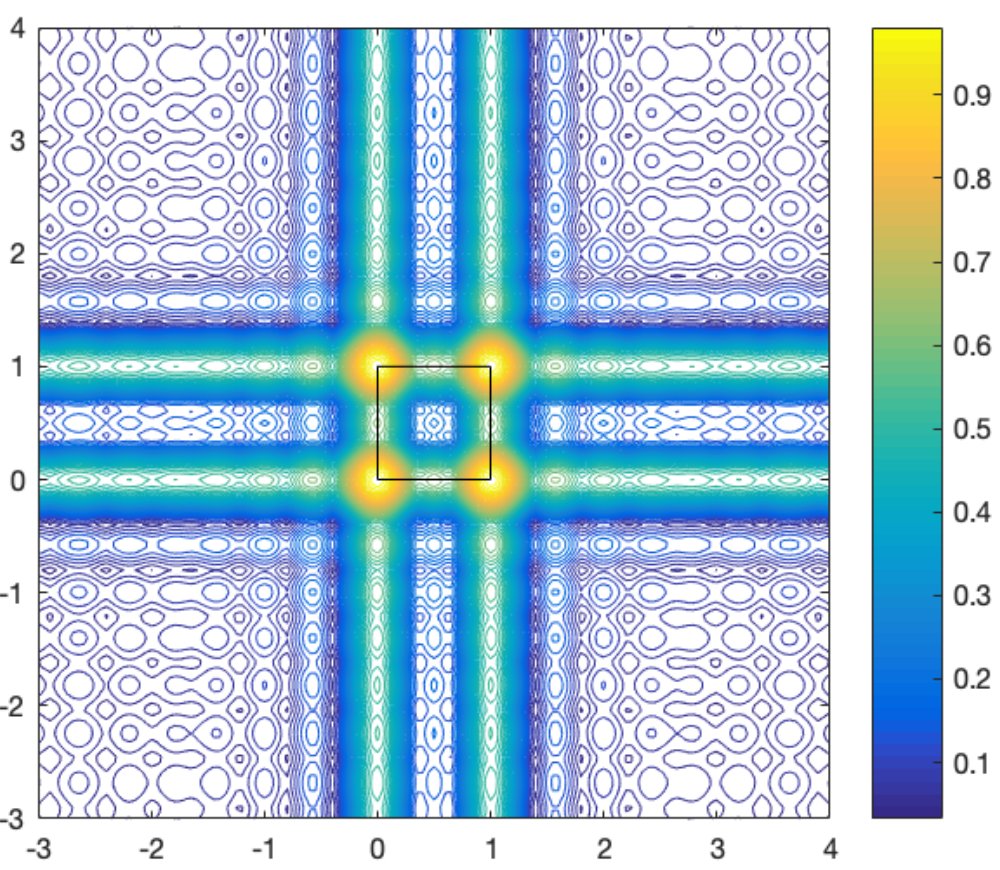}}
  \subfigure[\textbf{Twenty observation directions.}]{
    \includegraphics[width=1.5in]{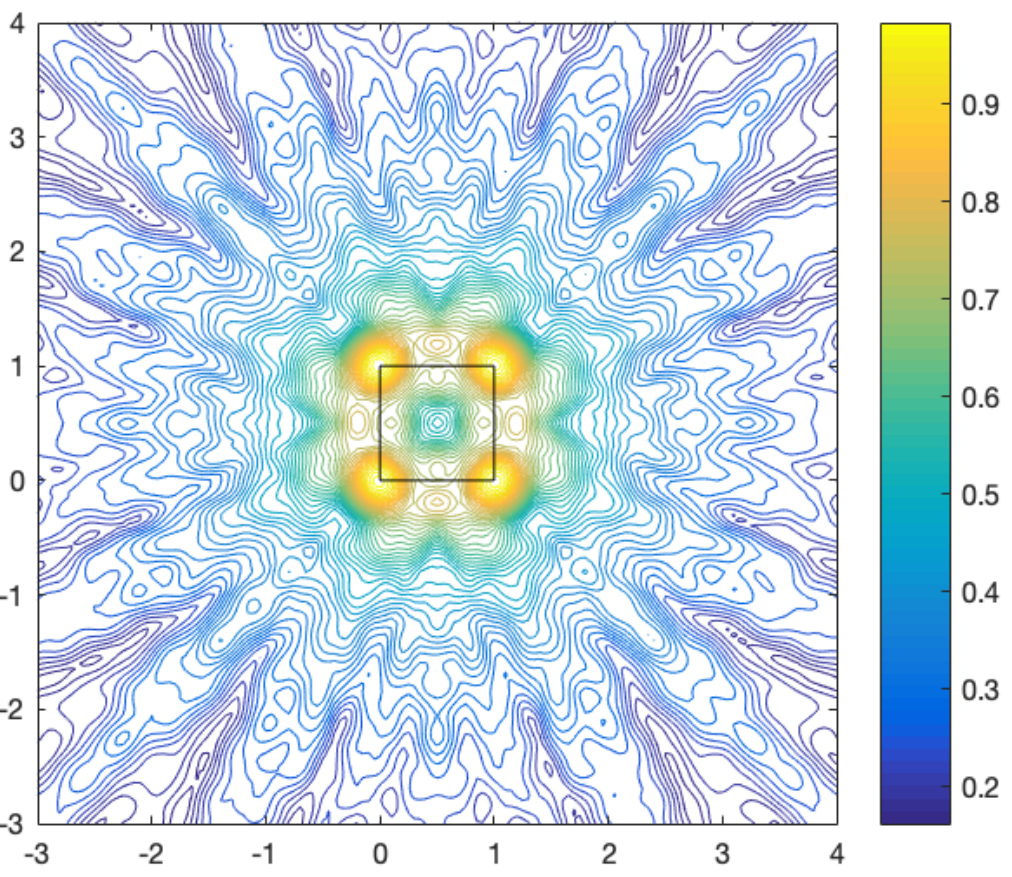}}
\caption{{\bf Strip Reconstruction Scheme One} with different observation directions for cubic support (${\bf S_1}$).}
\label{square1}
\end{figure}

\subsection{{\bf Strip Reconstruction Scheme One} for ball support (${\bf S_2}$)}
In this example, we consider the ball support reconstruction with $1, 2$ and $20$ observation directions. The phenomenon is similar to the previous example. We give the results in Fig. \ref{ball1}. The support is clearly reconstructed in Fig. \ref{ball1} (c). \\

\begin{figure}[htbp]
  \centering
  \subfigure[\textbf{One observation direction.}]{
    \includegraphics[width=1.5in]{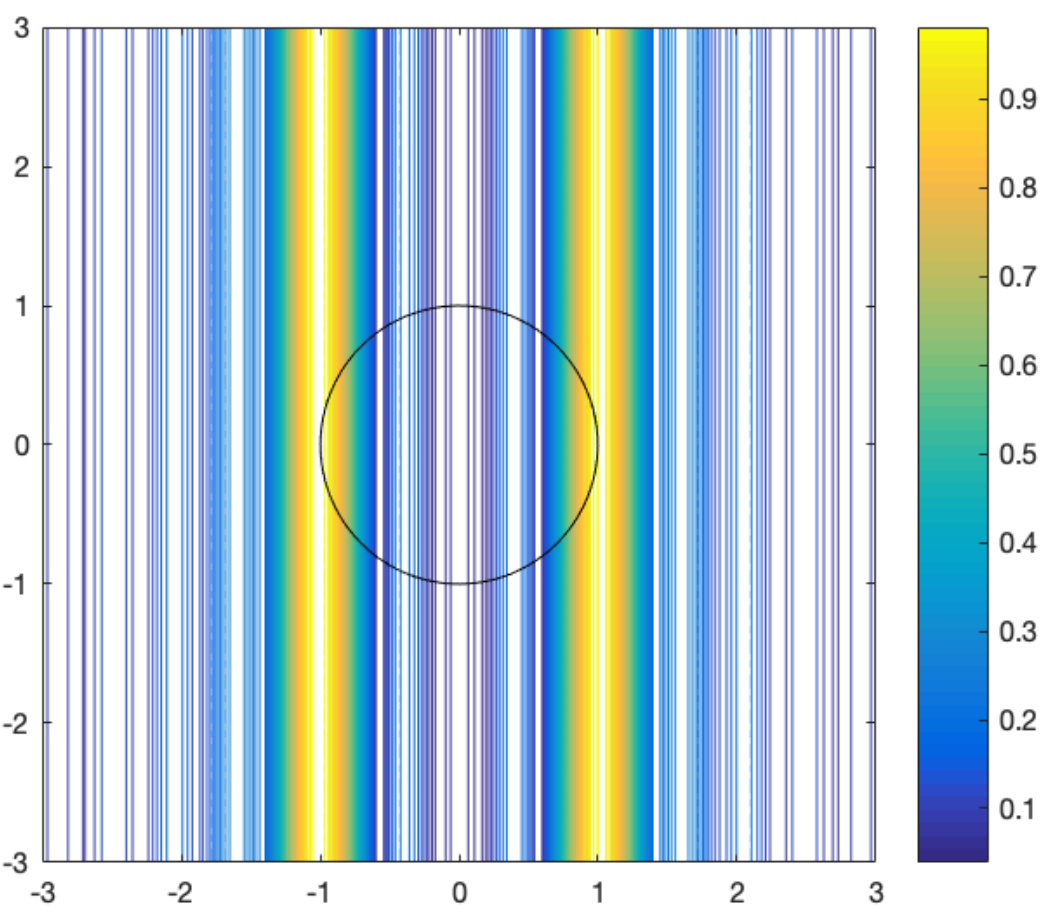}}
  \subfigure[\textbf{Two observation directions.}]{
    \includegraphics[width=1.5in]{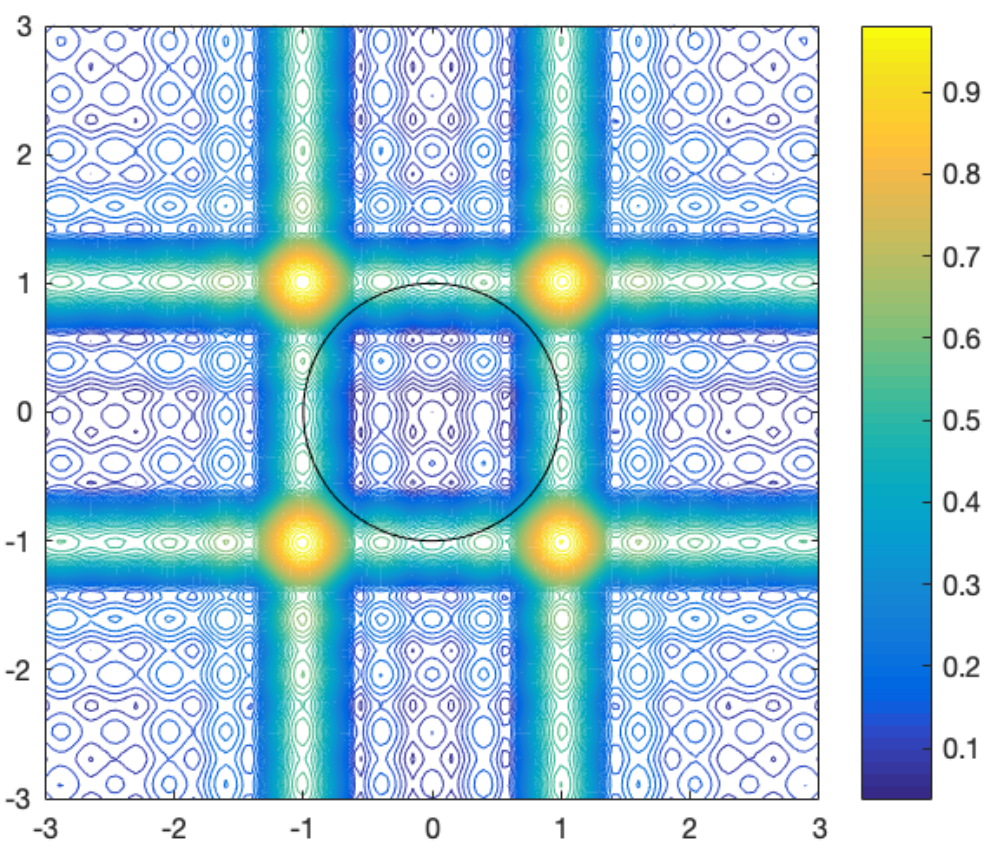}}
  \subfigure[\textbf{Twenty observation directions.}]{
    \includegraphics[width=1.5in]{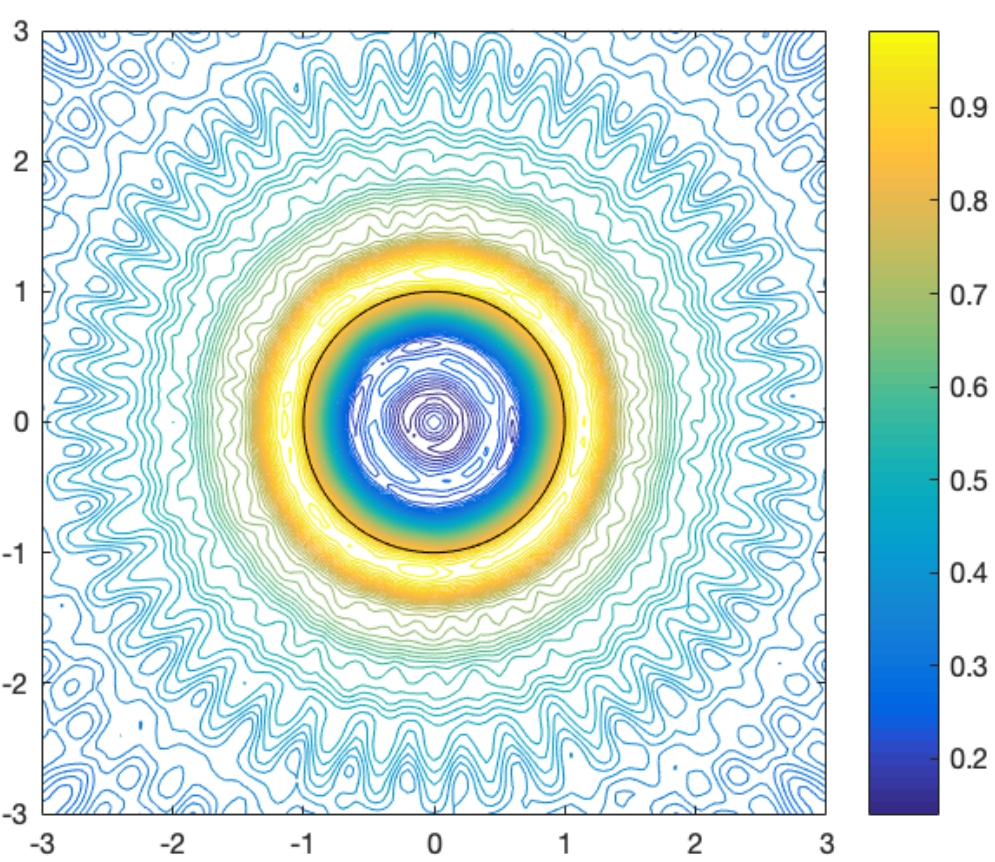}}
\caption{{\bf Strip Reconstruction Scheme One} with different observation directions for ball support (${\bf S_2}$).}
\label{ball1}
\end{figure}

\subsection{{\bf Strip Reconstruction Scheme One} for other supports (${\bf S_3}, {\bf S_4}, {\bf S_5}$)}
In this example, we consider the support reconstructions for ${\bf S_3}, {\bf S_4}$ and ${\bf S_5}$ with  $20$ observation directions. In Fig. \ref{S3S4} (a), the cubic and ball are both well constructed. The Fig. \ref{S3S4} (b) gives the L-shaped domain. In Fig. \ref{S5}, we give the ${\bf x-y}$ projection and ${\bf y-z}$ projection of the cuboid.

\begin{figure}[htbp]
  \centering
  \subfigure[\textbf{Cubic+ball.}]{
    \includegraphics[width=1.5in]{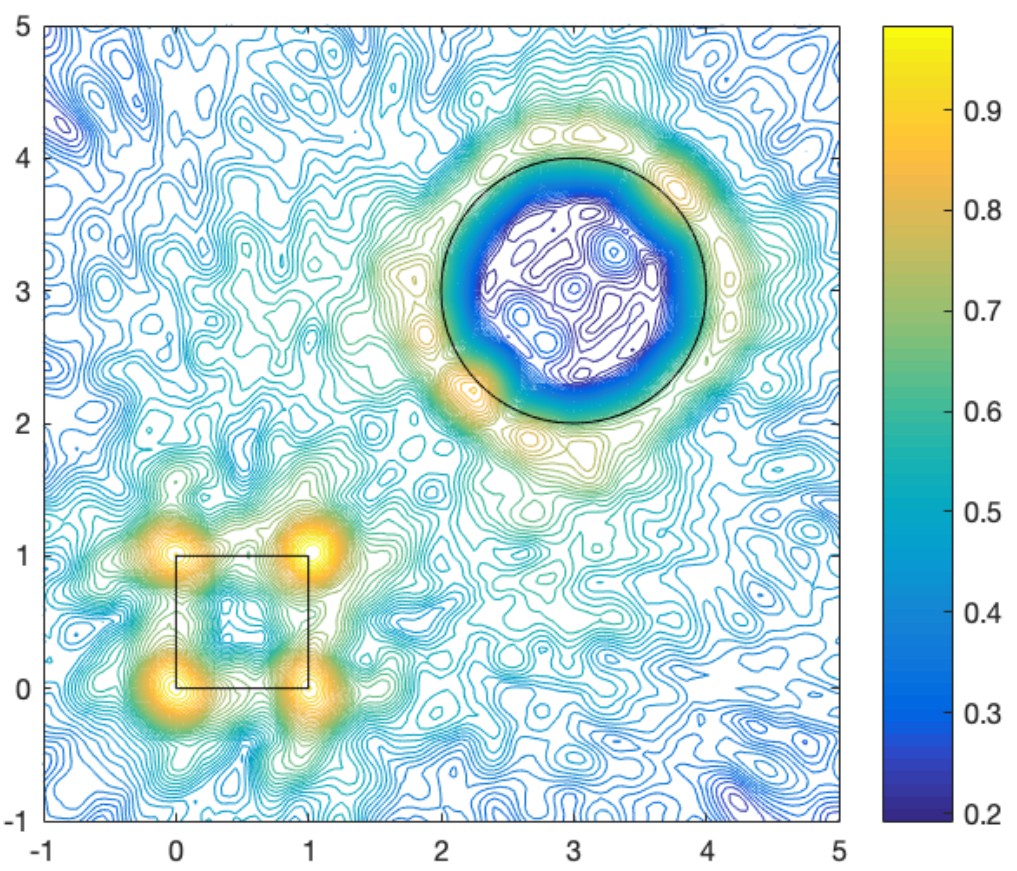}}
  \subfigure[\textbf{L-shaped domain.}]{
    \includegraphics[width=1.5in]{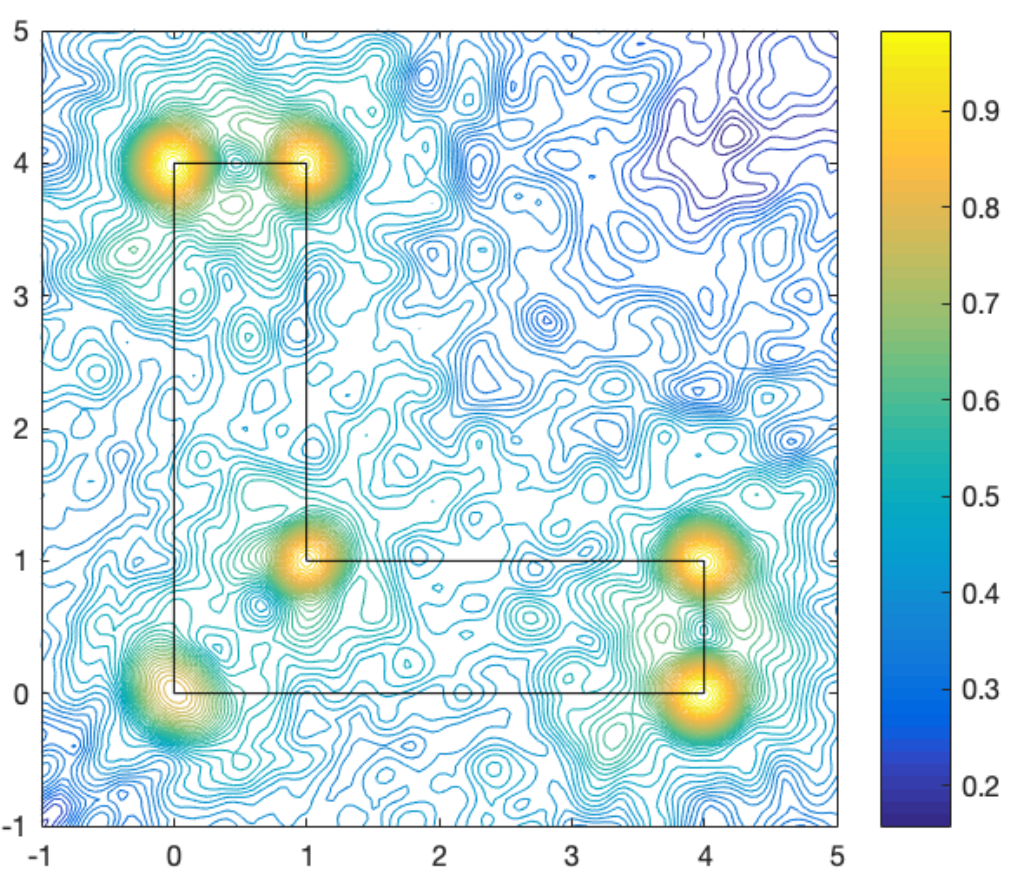}}
\caption{{\bf Strip Reconstruction Scheme One} with 20 observation directions for supports ${\bf S_3}$ and ${\bf S_4}$.}
\label{S3S4}
\end{figure}

\begin{figure}[htbp]
  \centering
  \subfigure[\textbf{$x-y$ projection.}]{
    \includegraphics[width=1.5in]{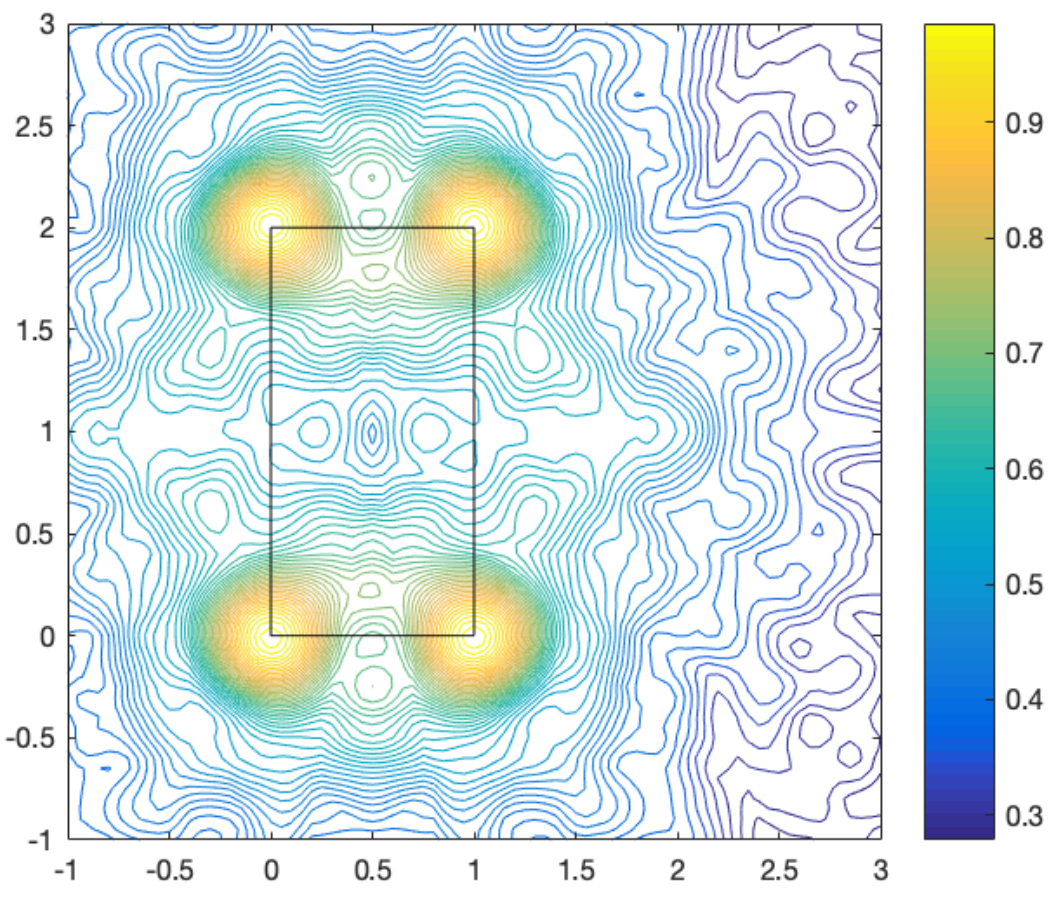}}
  \subfigure[\textbf{$y-z$ projection.}]{
    \includegraphics[width=1.5in]{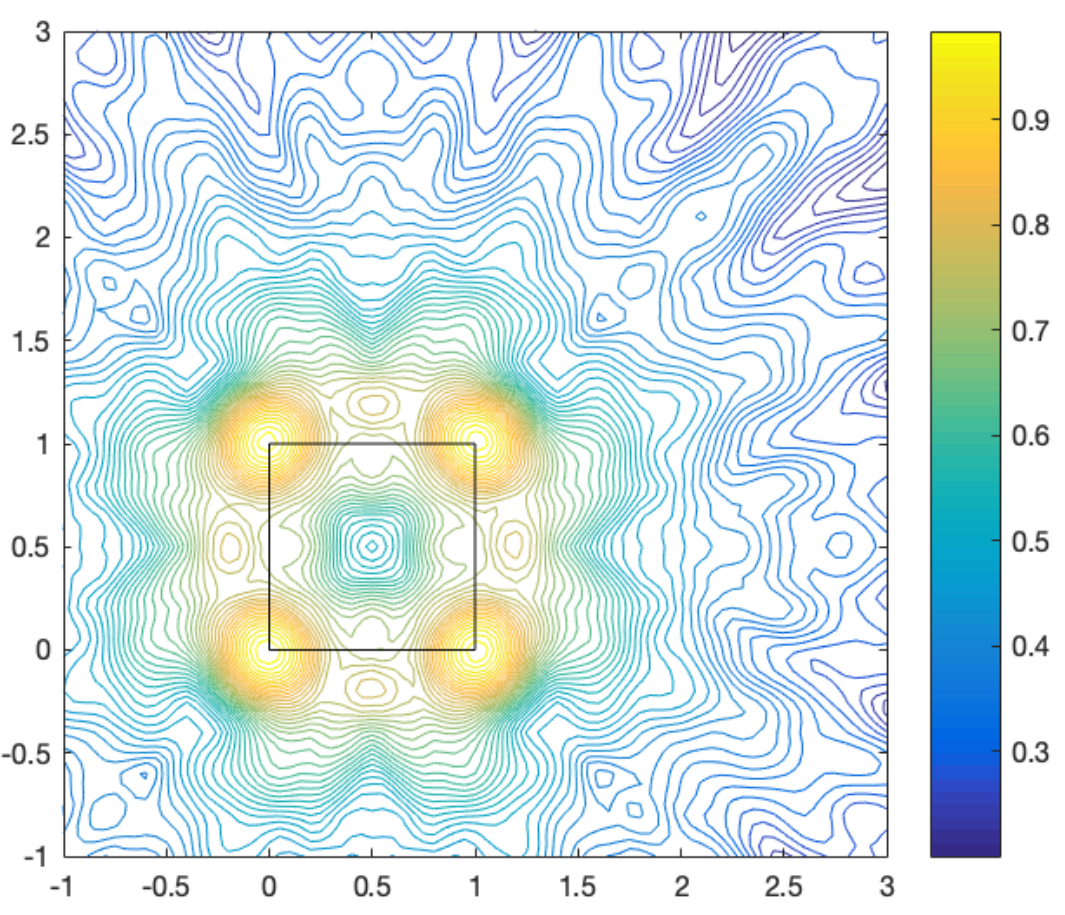}}
\caption{{\bf Strip Reconstruction Scheme One} with 20 observation directions for cuboid support ${\bf S_5}$.}
\label{S5}
\end{figure}

\subsection{{\bf Strip Reconstruction Scheme Two} for cubic support (${\bf S_1}$)}
We first consider the case of one observation direction using different $\bf {z_0}$. In Fig. \ref{square2}, we plot the indicators using $\bf\hx_0= (1,0,0)^{\rm T}$ and three reference points ${\bf z_0} = (0.5,2,0)^{\rm T}, {\bf z_0}  = (2,2,0)^{\rm T}$ and ${\bf z_0}  = (4,4,0)^{\rm T}$. The picture clearly shows that the source support and the corresponding point symmetric domain (with respect to $\bf z_0$) lies in a strip, which is perpendicular to the observation direction.
\begin{figure}[htbp]
  \centering
  \subfigure[\textbf{${\bf z_0} = (0.5,2,0)^{\rm T}$.}]{
    \includegraphics[width=1.5in, height=1.3in]{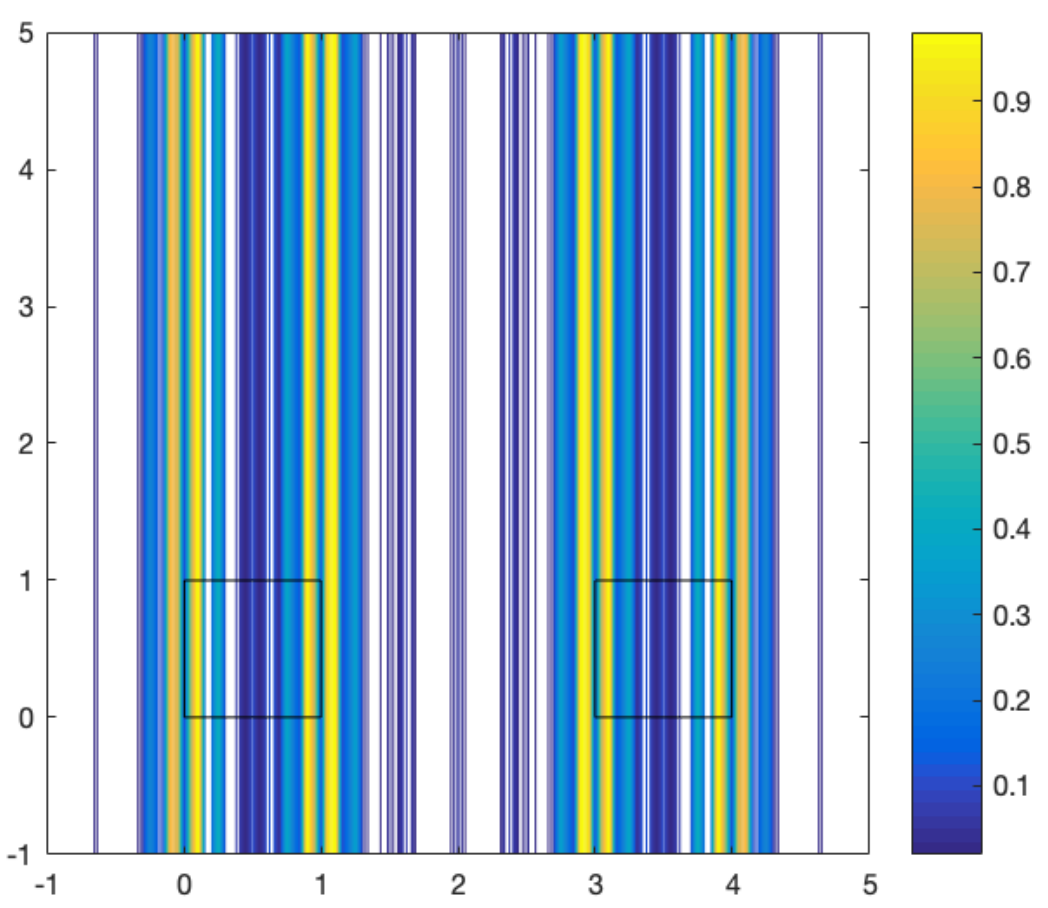}}
  \subfigure[\textbf{${\bf z_0} = (2,2,0)^{\rm T}$.}]{
    \includegraphics[width=1.5in]{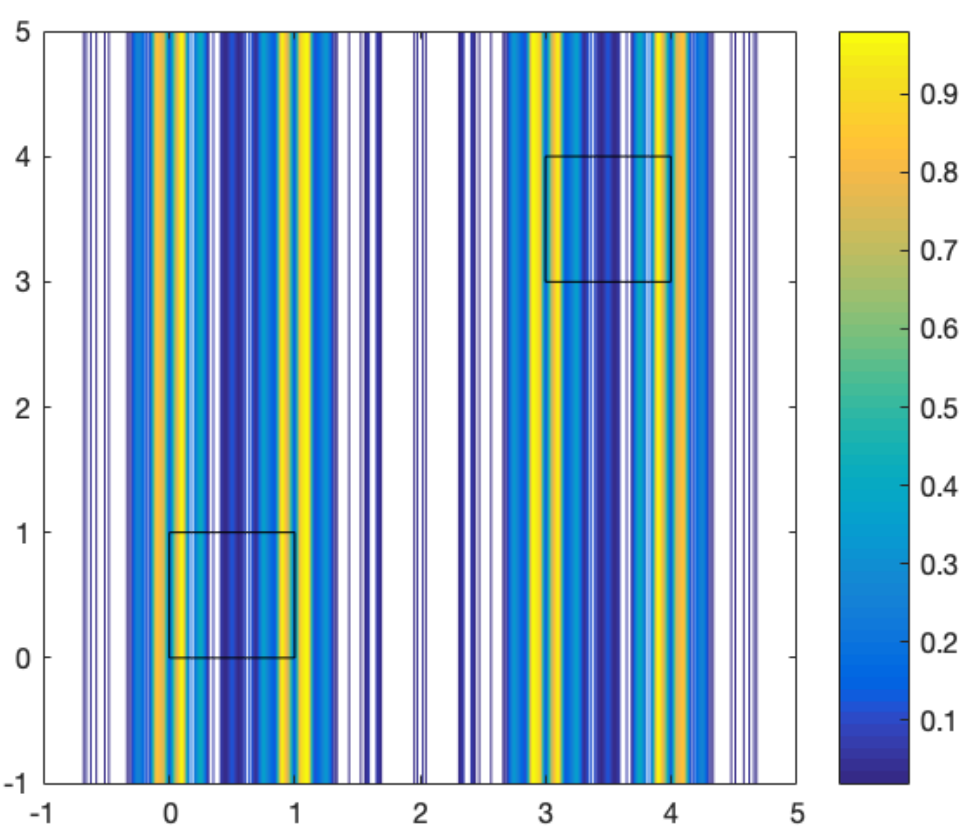}}
  \subfigure[\textbf{${\bf z_0} = (4,4,0)^{\rm T}$.}]{
    \includegraphics[width=1.5in]{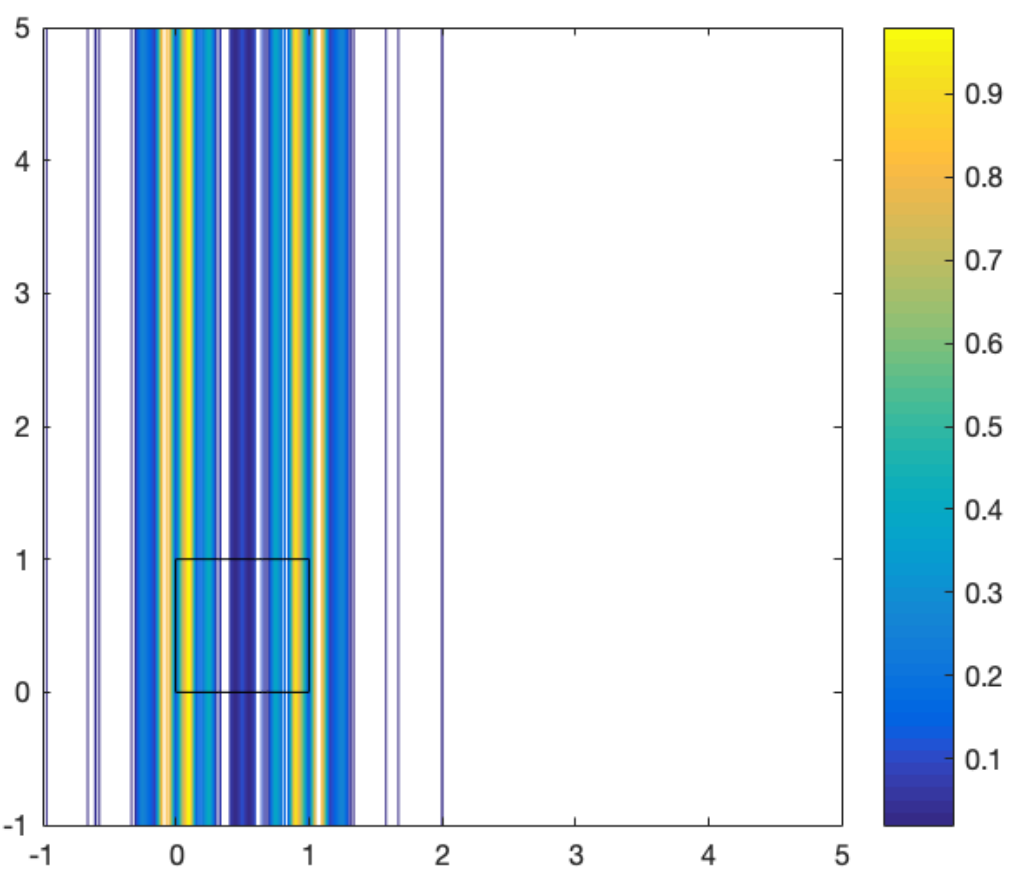}}
\caption{{\bf Strip Reconstruction Scheme Two} with different  $\bf {z_0}$ and one observation direction with cubic support (${\bf S_1}$).}
\label{square2}
\end{figure}

Next we consider two observation directions $(1,0,0)^{\rm T}$ and $(0,1,0)^{\rm T}$, we plot the indicators in Fig. \ref{square3}.  The source support also can be located in the cross sections of the two strips. To find the correct source support, one may consider choosing a reference point far away from the sampling domain (as shown in Figure \ref{square3}(c)), or using multiple observation directions.

\begin{figure}[htbp]
  \centering
  \subfigure[\textbf{${\bf z_0} = (0.5,2,0)^{\rm T}$.}]{
    \includegraphics[width=1.5in, height=1.3in]{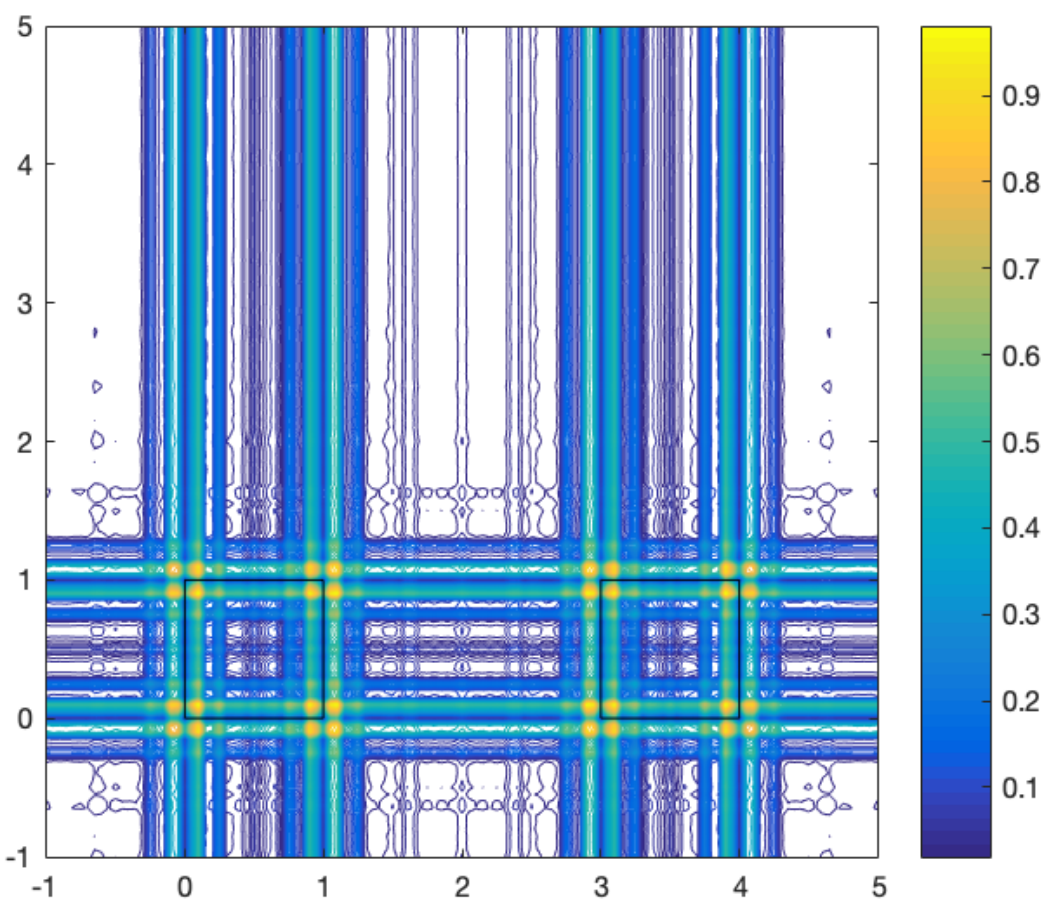}}
  \subfigure[\textbf{${\bf z_0} = (2,2,0)^{\rm T}$.}]{
    \includegraphics[width=1.5in]{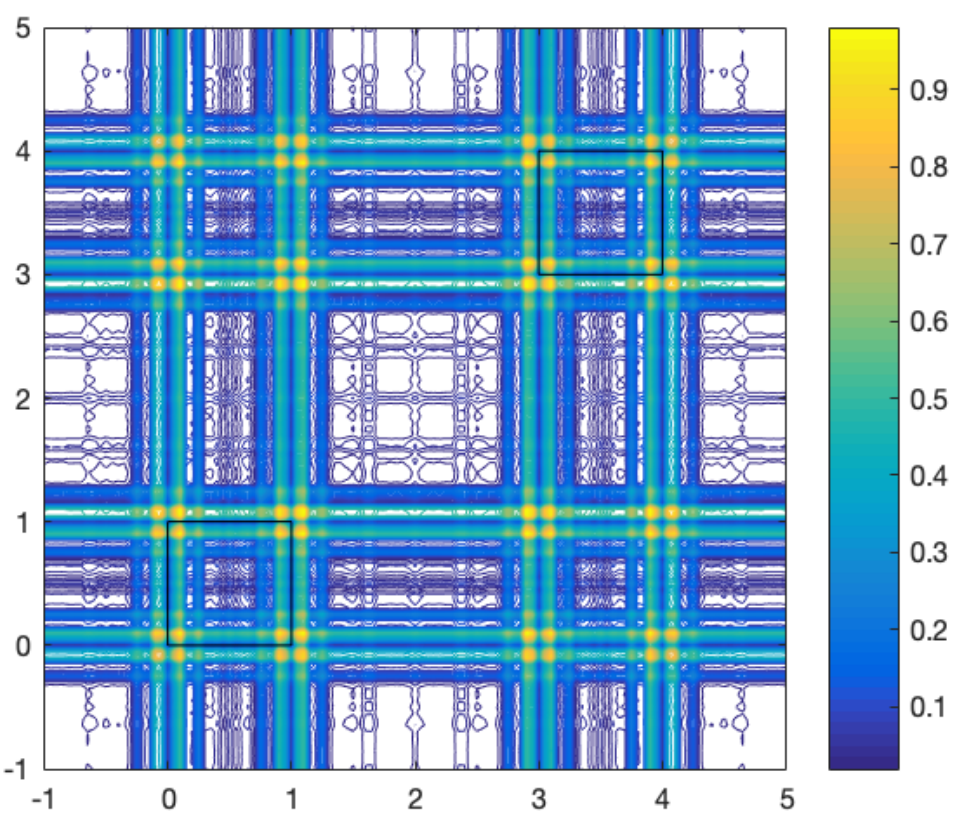}}
  \subfigure[\textbf{${\bf z_0} = (4,4,0)^{\rm T}$.}]{
    \includegraphics[width=1.5in]{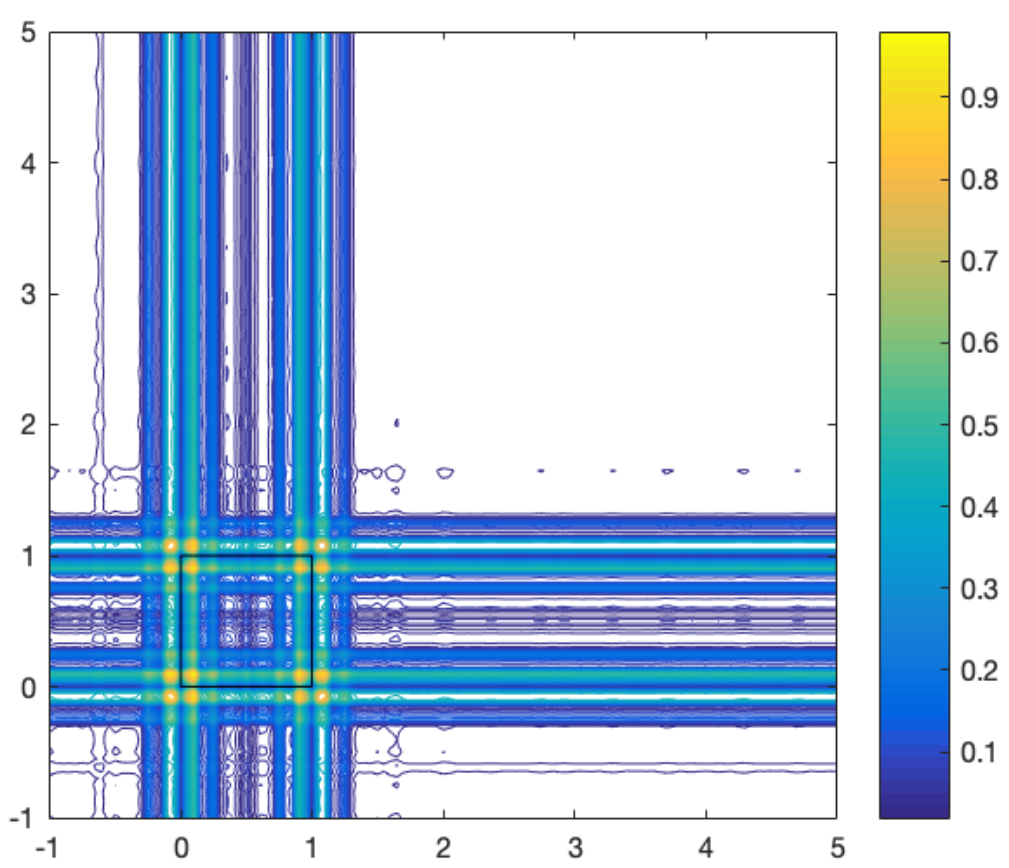}}
\caption{{\bf Strip Reconstruction Scheme Two} with different  $\bf {z_0}$ and two observation directions with cubic support (${\bf S_1}$).}
\label{square3}
\end{figure}

Now we use 20 observation directions. Fig. \ref{square4} gives the results for
cubic support with different $\bf {z_0}$. The locations and sizes
of support are reconstructed correctly.

\begin{figure}[htbp]
  \centering
  \subfigure[\textbf{${\bf z_0} = (0.5,2,0)^{\rm T}$.}]{
    \includegraphics[width=1.5in, height=1.3in]{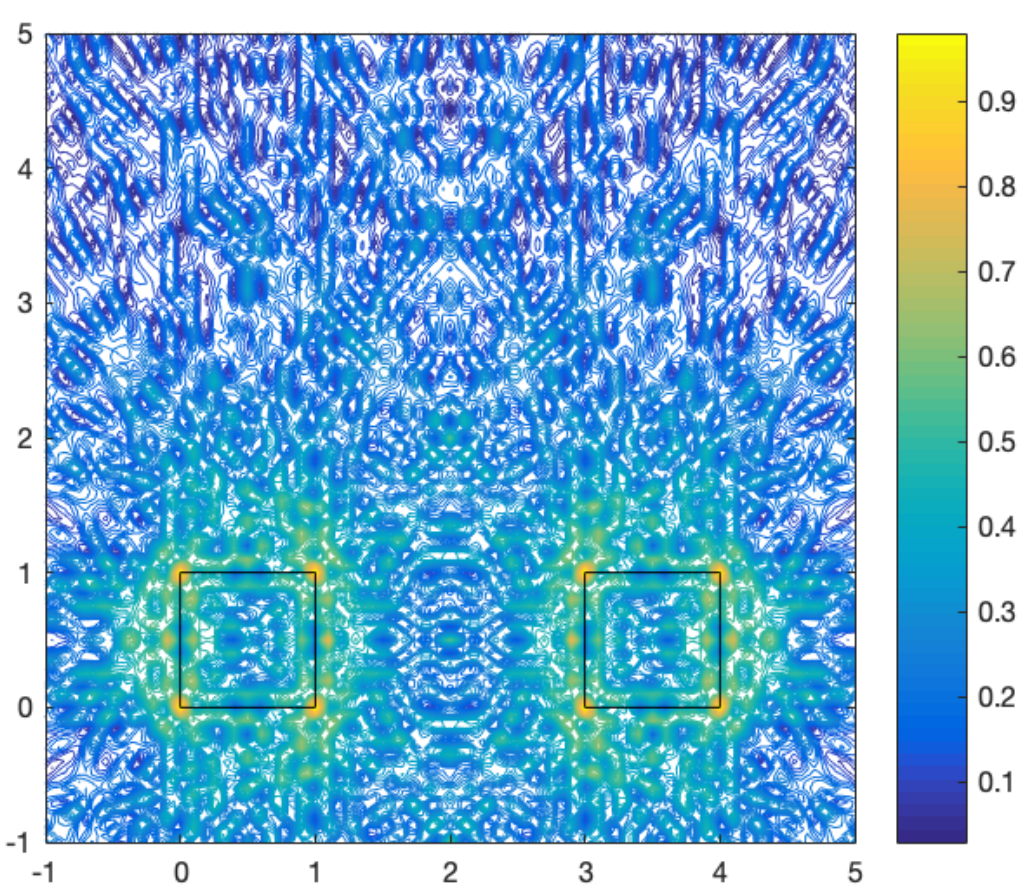}}
  \subfigure[\textbf{${\bf z_0} = (2,2,0)^{\rm T}$.}]{
    \includegraphics[width=1.5in]{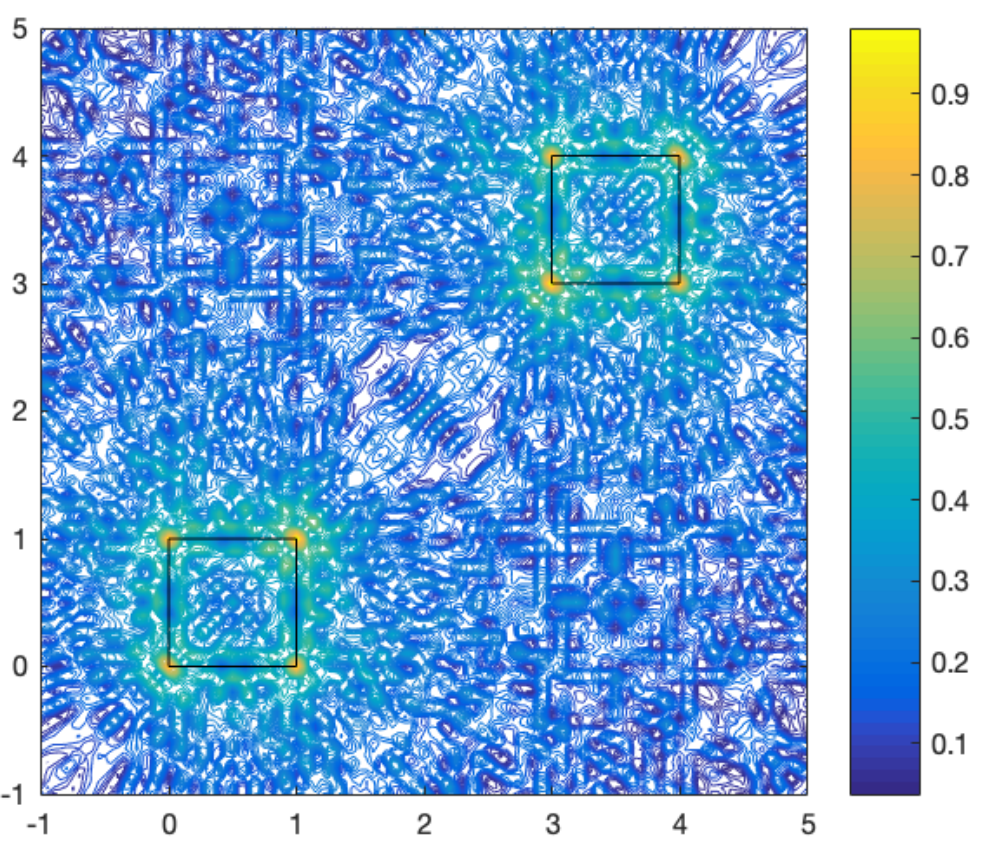}}
  \subfigure[\textbf{${\bf z_0} = (4,4,0)^{\rm T}$.}]{
    \includegraphics[width=1.5in]{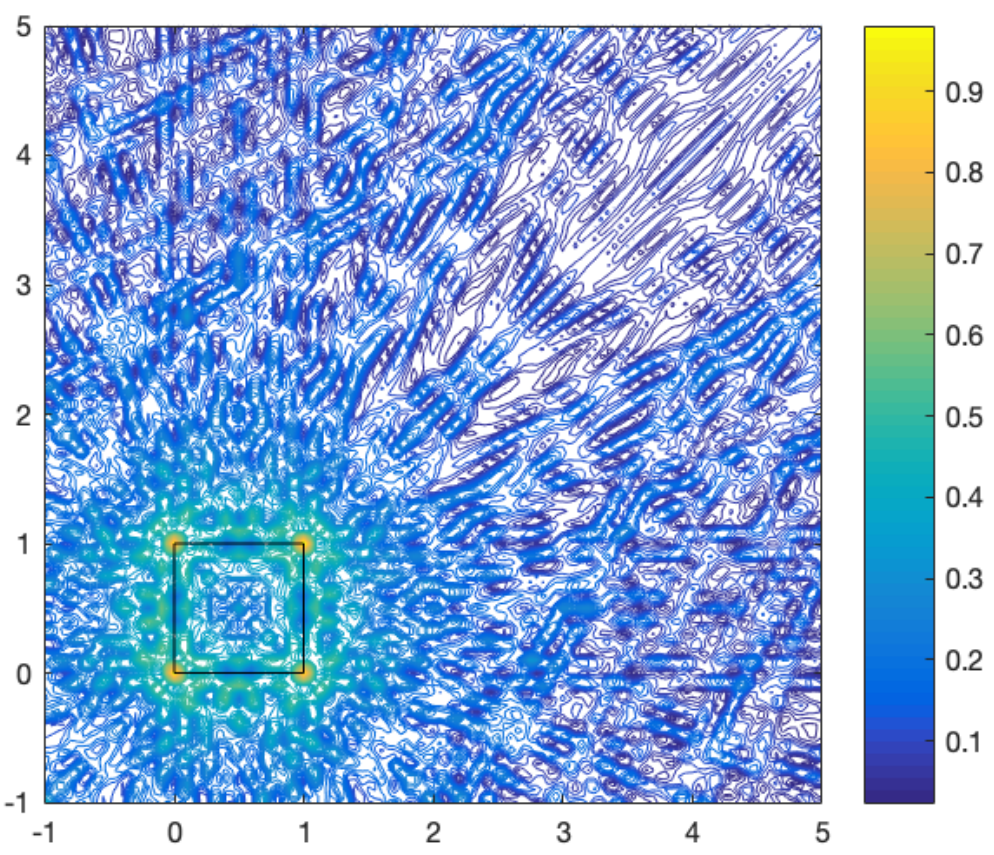}}
\caption{{\bf Strip Reconstruction Scheme Two} with different  $\bf {z_0}$ and twenty observation directions for cubic support (${\bf S_1}$).}
\label{square4}
\end{figure}

\subsection{{\bf Strip Reconstruction Scheme Two} for ball support (${\bf S_2}$)}
This example gives the results for ball support with three reference points ${\bf z_0} = (0,2,0)^{\rm T}, {\bf z_0}  = (2,2,0)^{\rm T}$ and ${\bf z_0}  = (4,4,0)^{\rm T}$.  Fig. \ref{ball2}, Fig. \ref{ball3} and Fig. \ref{ball4} give the indicators for one, two and twenty observation directions, respectively.

\begin{figure}[htbp]
  \centering
  \subfigure[\textbf{${\bf z_0} = (0,2,0)^{\rm T}$.}]{
    \includegraphics[width=1.5in, height=1.3in]{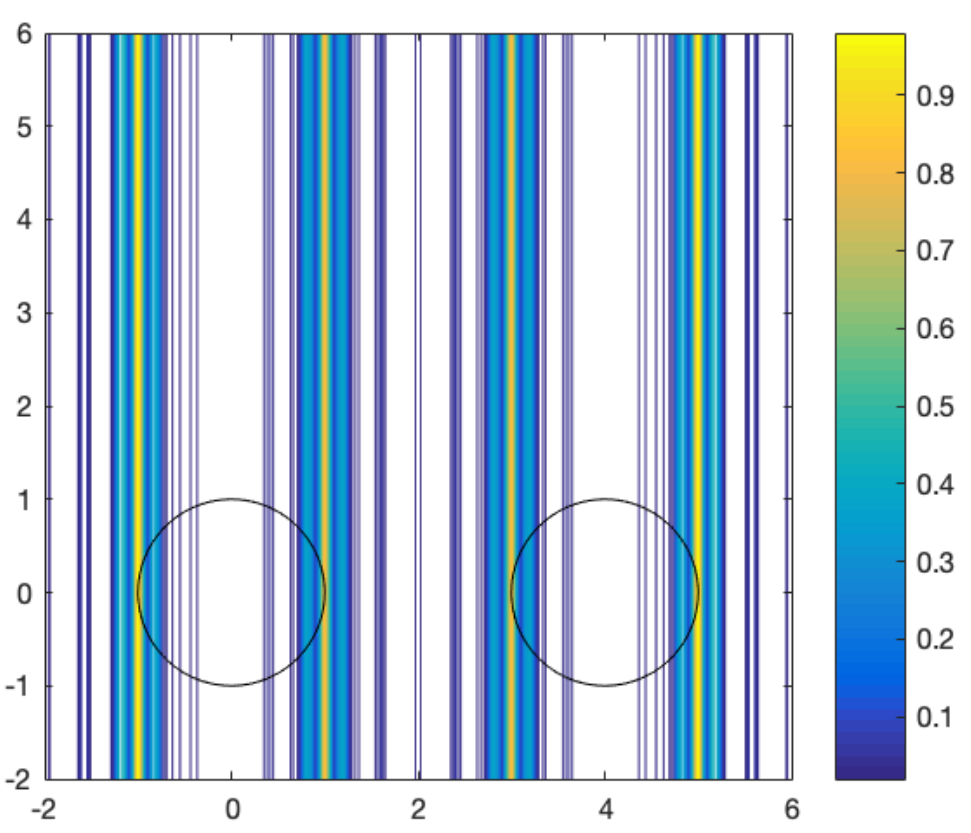}}
  \subfigure[\textbf{${\bf z_0} = (2,2,0)^{\rm T}$.}]{
    \includegraphics[width=1.5in]{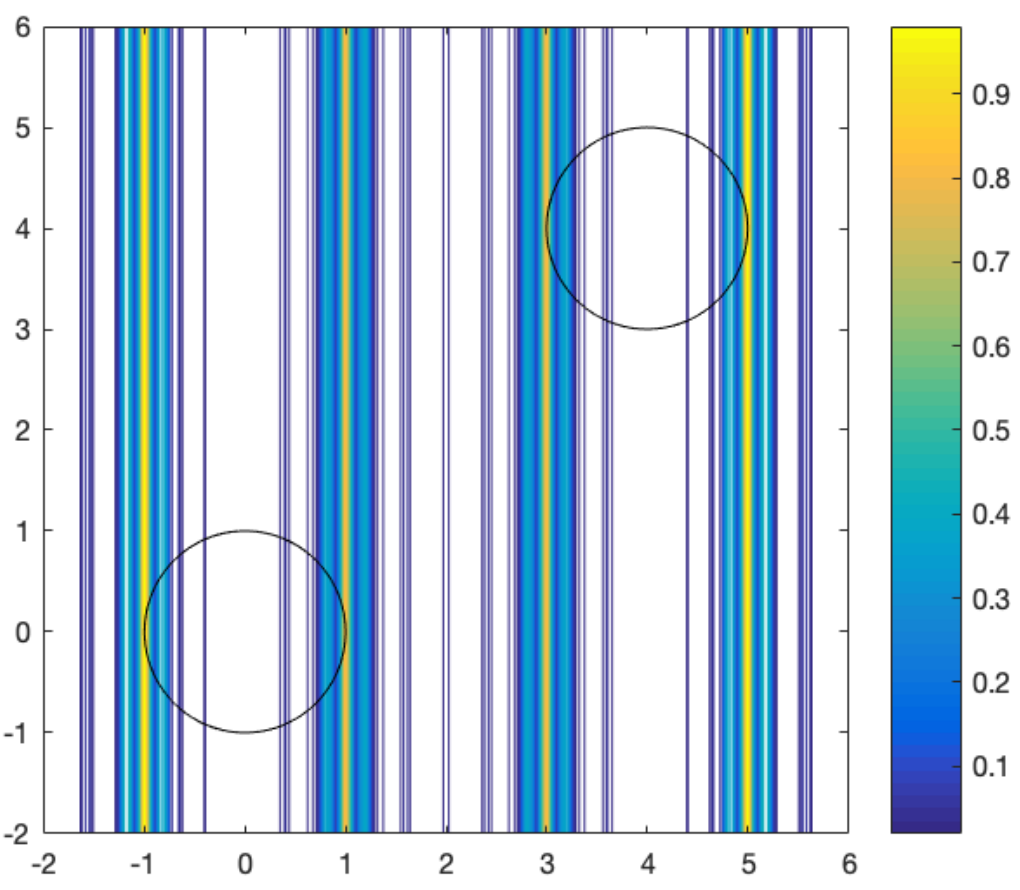}}
  \subfigure[\textbf{${\bf z_0} = (4,4,0)^{\rm T}$.}]{
    \includegraphics[width=1.5in]{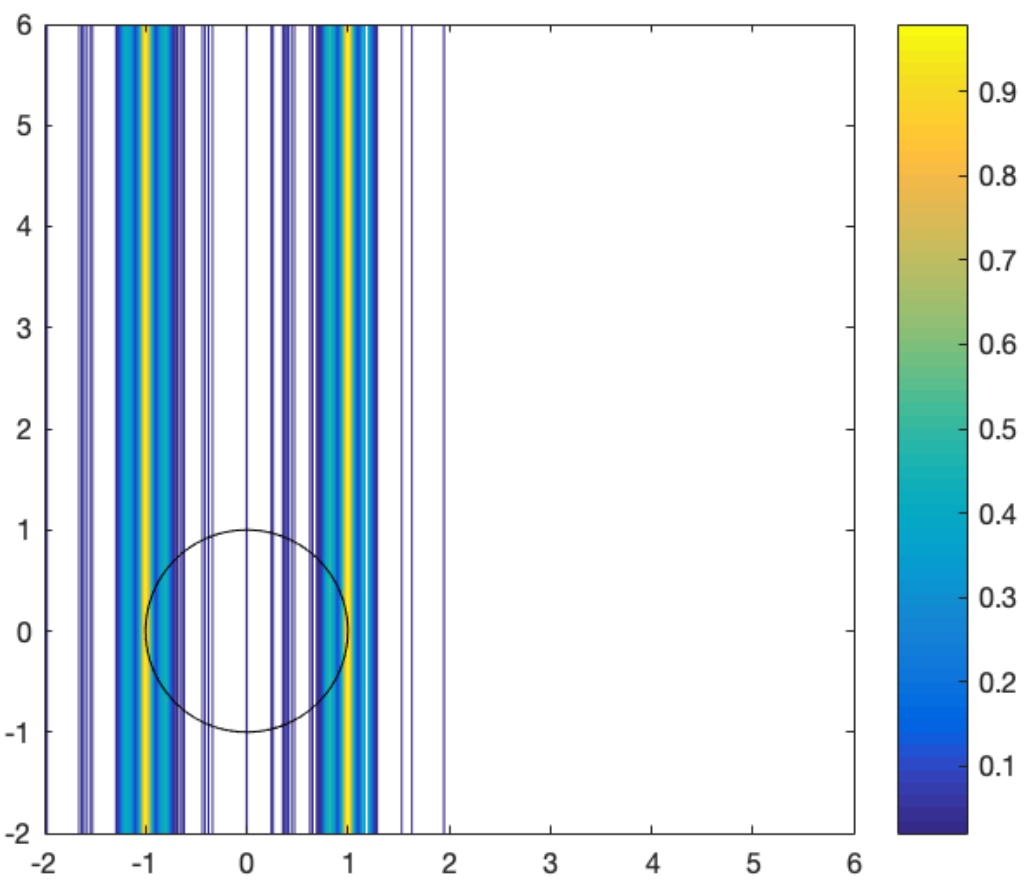}}
\caption{{\bf Strip Reconstruction Scheme Two} with different  $\bf {z_0}$ and one observation direction for ball support (${\bf S_2}$).}
\label{ball2}
\end{figure}

\begin{figure}[htbp]
  \centering
  \subfigure[\textbf{${\bf z_0} = (0,2,0)^{\rm T}$.}]{
    \includegraphics[width=1.5in, height=1.3in]{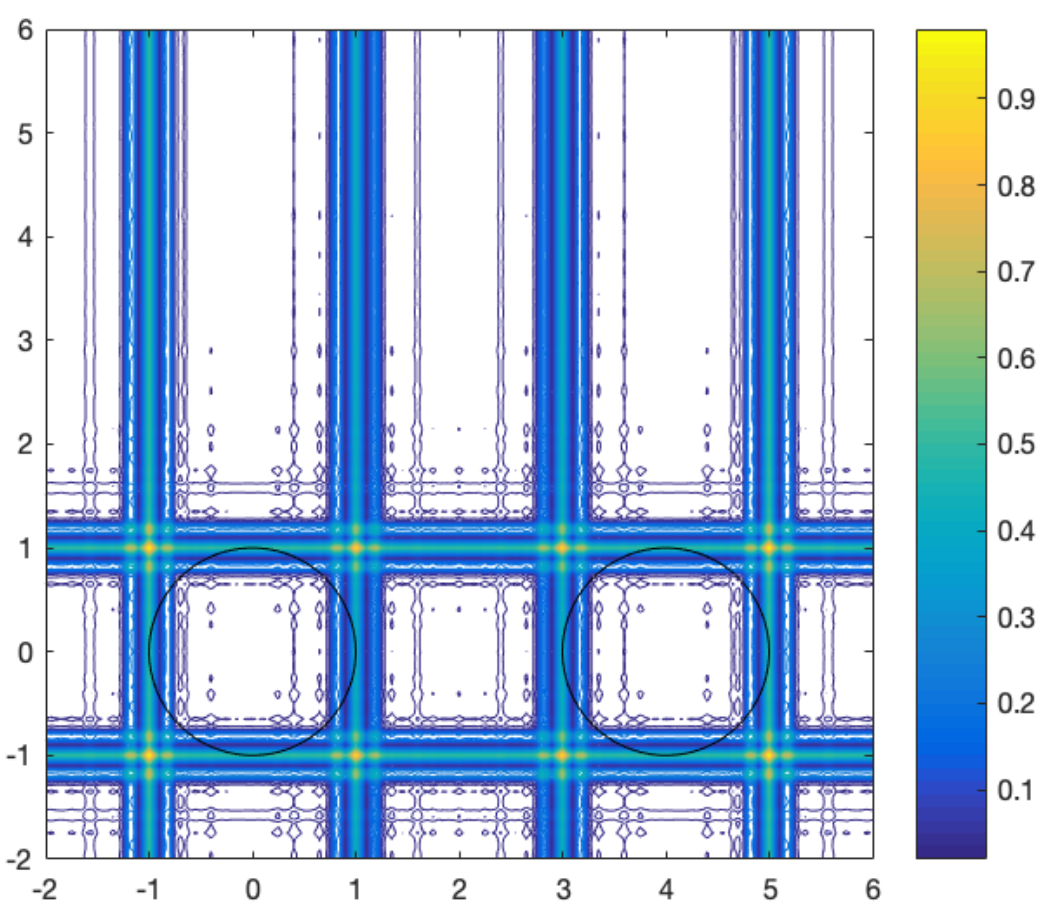}}
  \subfigure[\textbf{${\bf z_0} = (2,2,0)^{\rm T}$.}]{
    \includegraphics[width=1.5in]{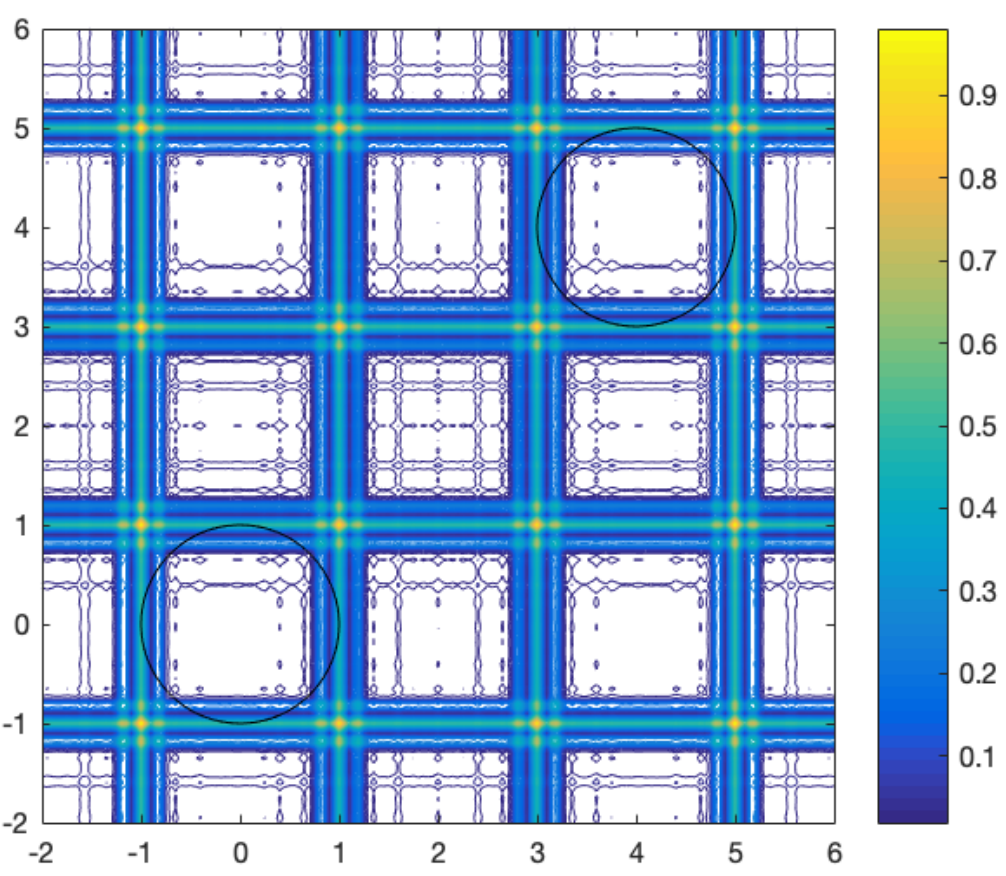}}
  \subfigure[\textbf{${\bf z_0} = (4,4,0)^{\rm T}$.}]{
    \includegraphics[width=1.5in]{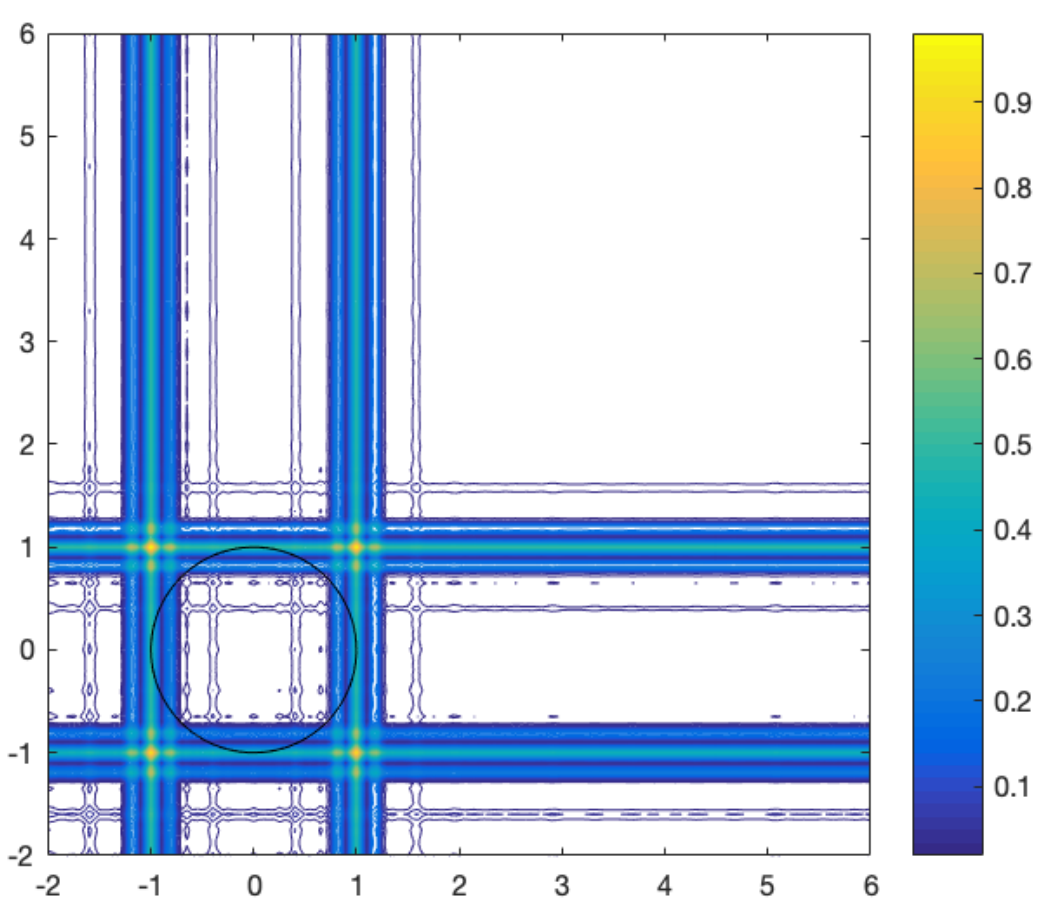}}
\caption{{\bf Strip Reconstruction Scheme Two} with different  $\bf {z_0}$ and two observation directions for ball support (${\bf S_2}$).}
\label{ball3}
\end{figure}

\begin{figure}[htbp]
  \centering
  \subfigure[\textbf{${\bf z_0} = (0,2,0)^{\rm T}$.}]{
    \includegraphics[width=1.5in, height=1.3in]{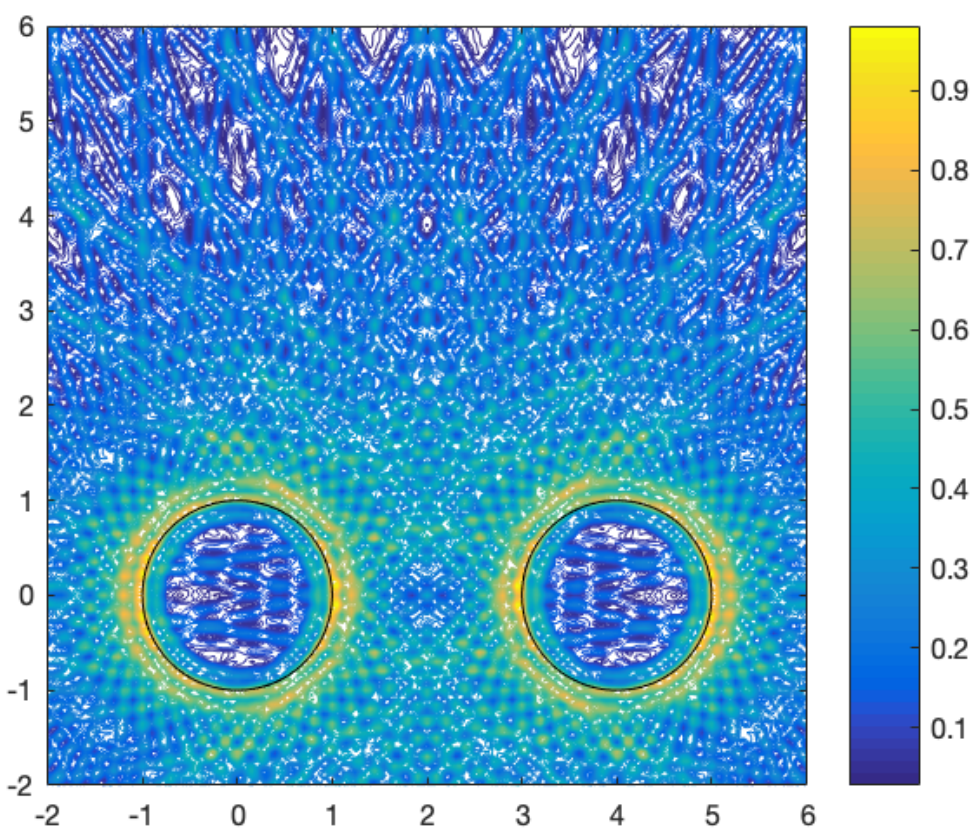}}
  \subfigure[\textbf{${\bf z_0} = (2,2,0)^{\rm T}$.}]{
    \includegraphics[width=1.5in]{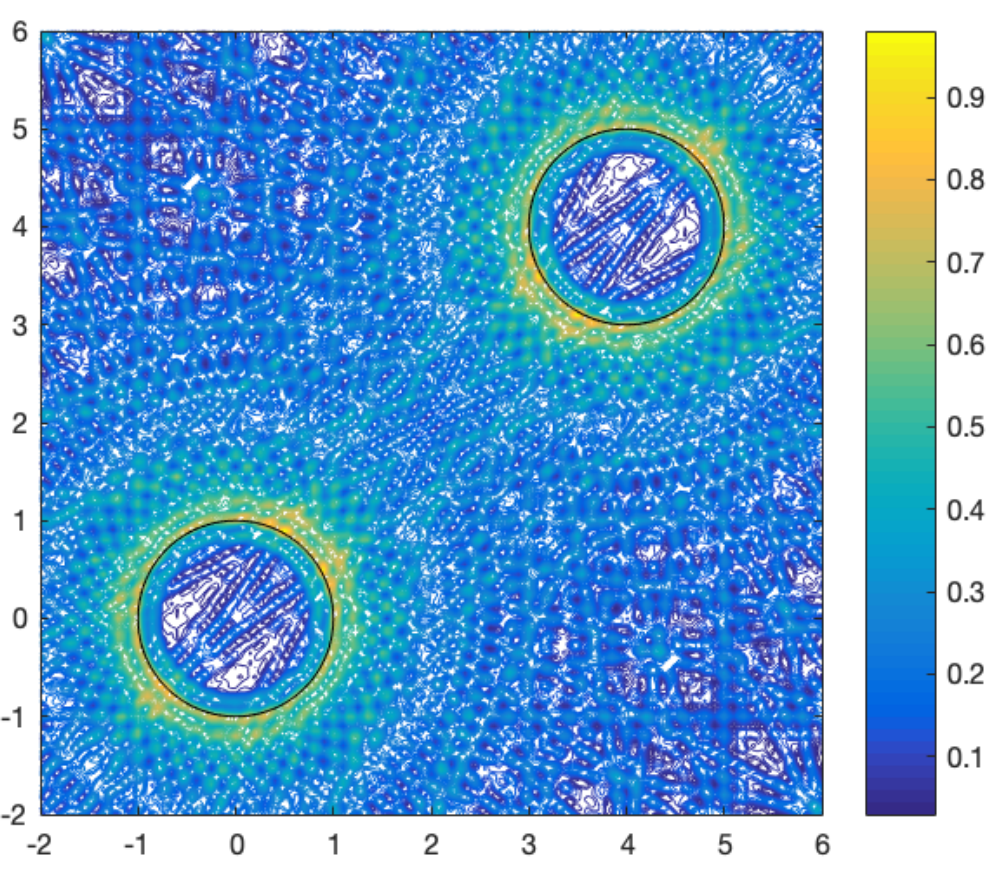}}
  \subfigure[\textbf{${\bf z_0} = (4,4,0)^{\rm T}$.}]{
    \includegraphics[width=1.5in]{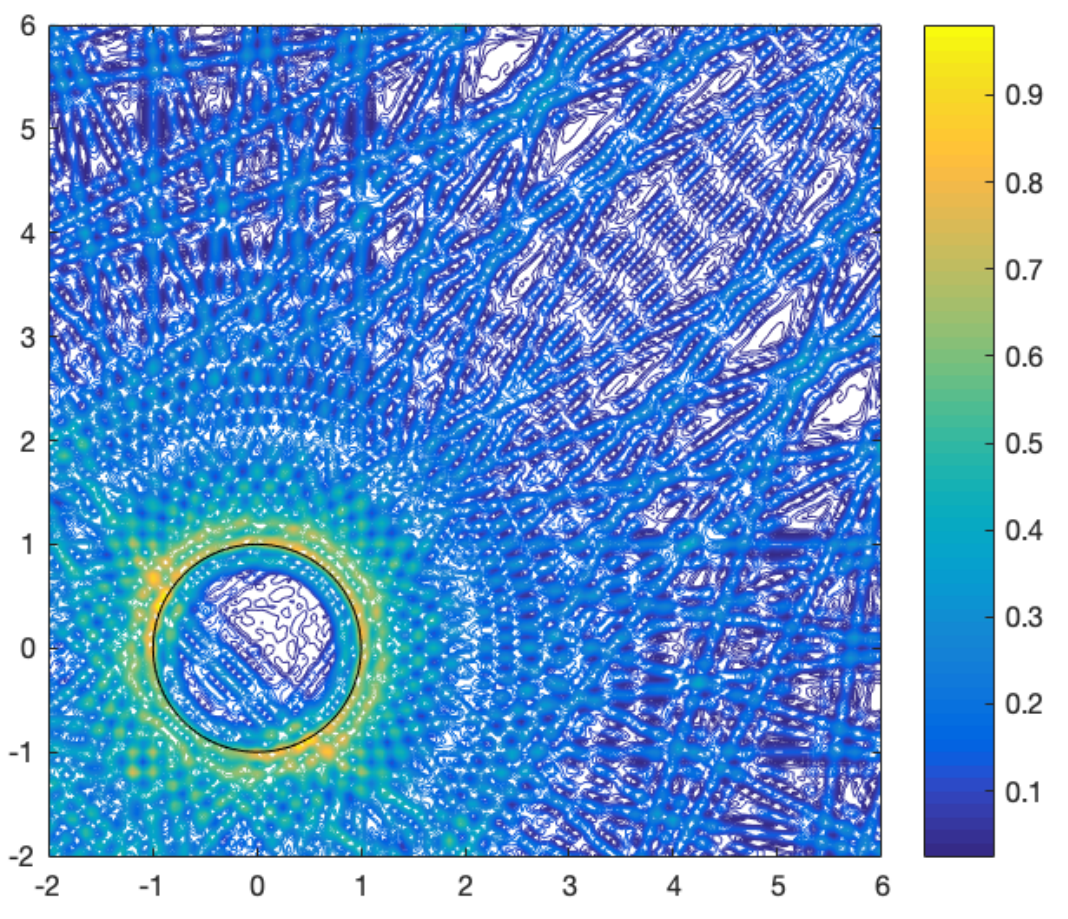}}
\caption{{\bf Strip Reconstruction Scheme Two} with different  $\bf {z_0}$ and twenty observation directions for ball support (${\bf S_2}$).}
\label{ball4}
\end{figure}

\subsection{The validity of the phase retrieval scheme}
This example is designed to check the phase retrieval scheme
proposed in the previous section. The underlying scatterer is cubic.
In Figures \ref{error1} and \ref{error2}, we compare the phase retrieval data with the exact one, the real part of $ \bf m\cdot\bE^{\infty}(\mathbf{\hx}_0,k;\mathbf{J})$
at a fixed direction $\bf\hx_0=(1,0,0)^{\rm T}$ is given. We observe that our phase retrieval scheme is very robust to noise.

\begin{figure}[htbp]
  \centering
  \subfigure[\textbf{ No noise.}]{
    \includegraphics[height=1.3in,width=1.5in]{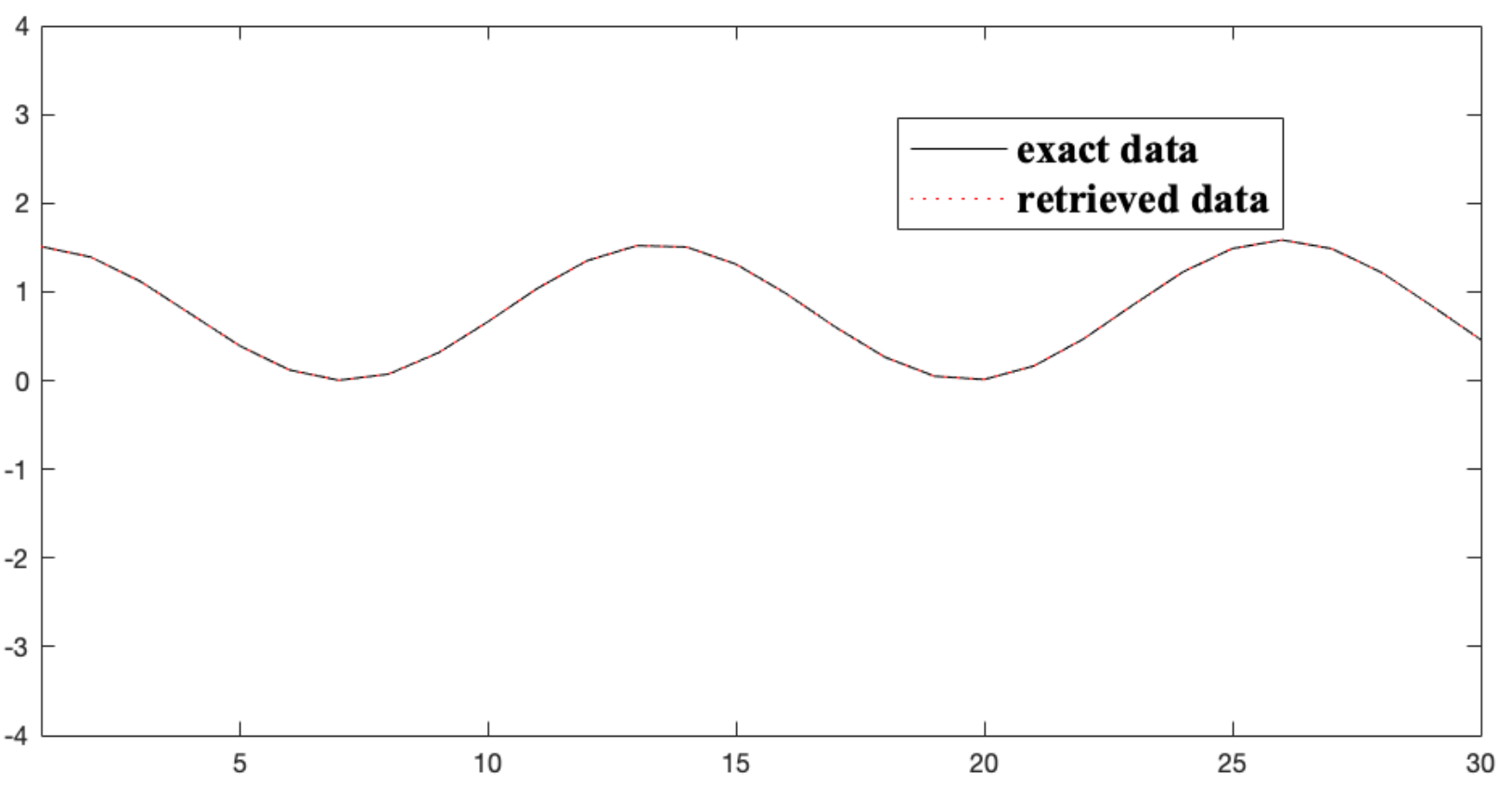}}
  \subfigure[\textbf{$10\%$ noise.}]{
    \includegraphics[height=1.3in,width=1.5in]{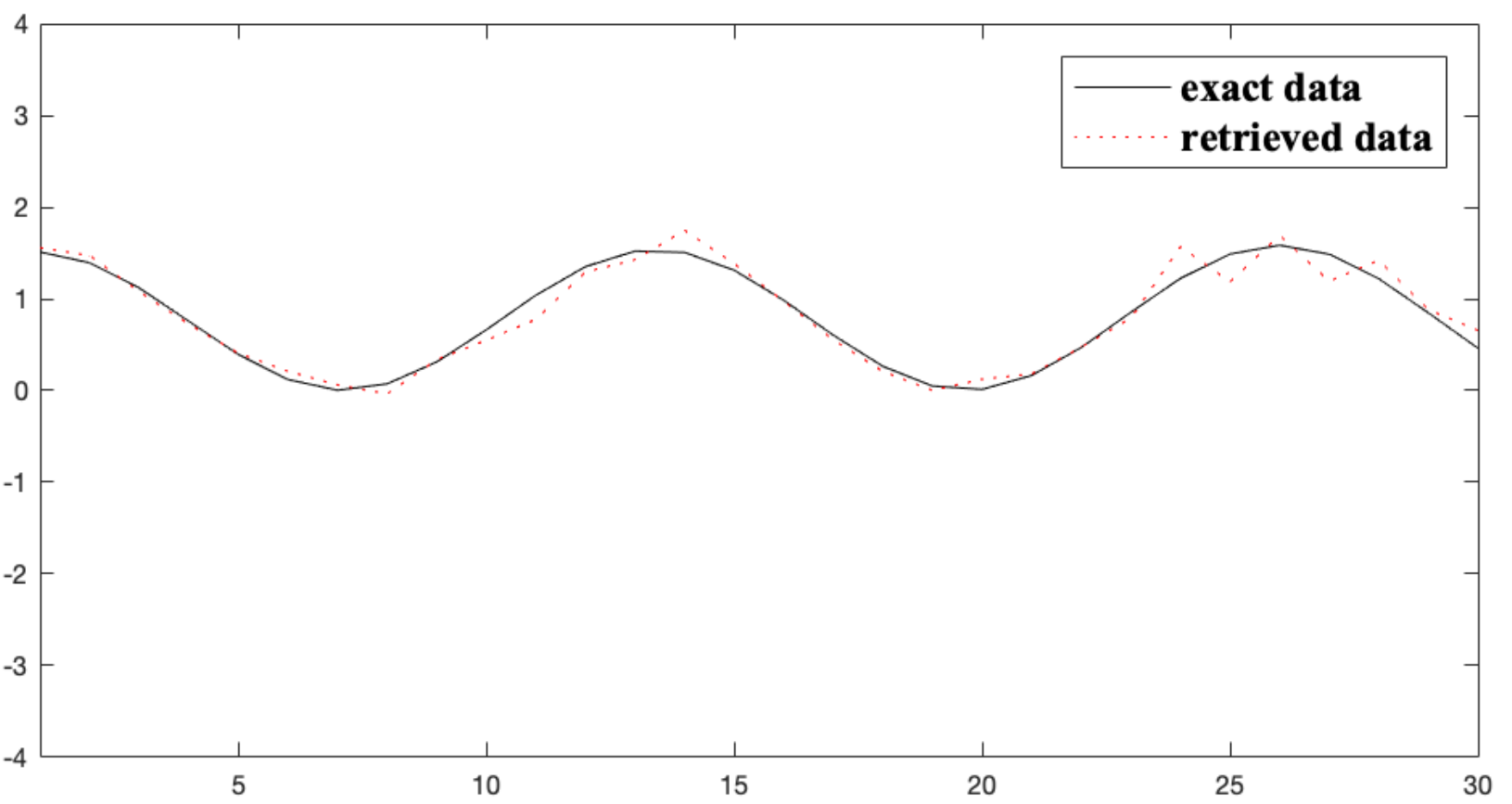}}
  \subfigure[\textbf{$30\%$ noise.}]{
    \includegraphics[height=1.3in,width=1.5in]{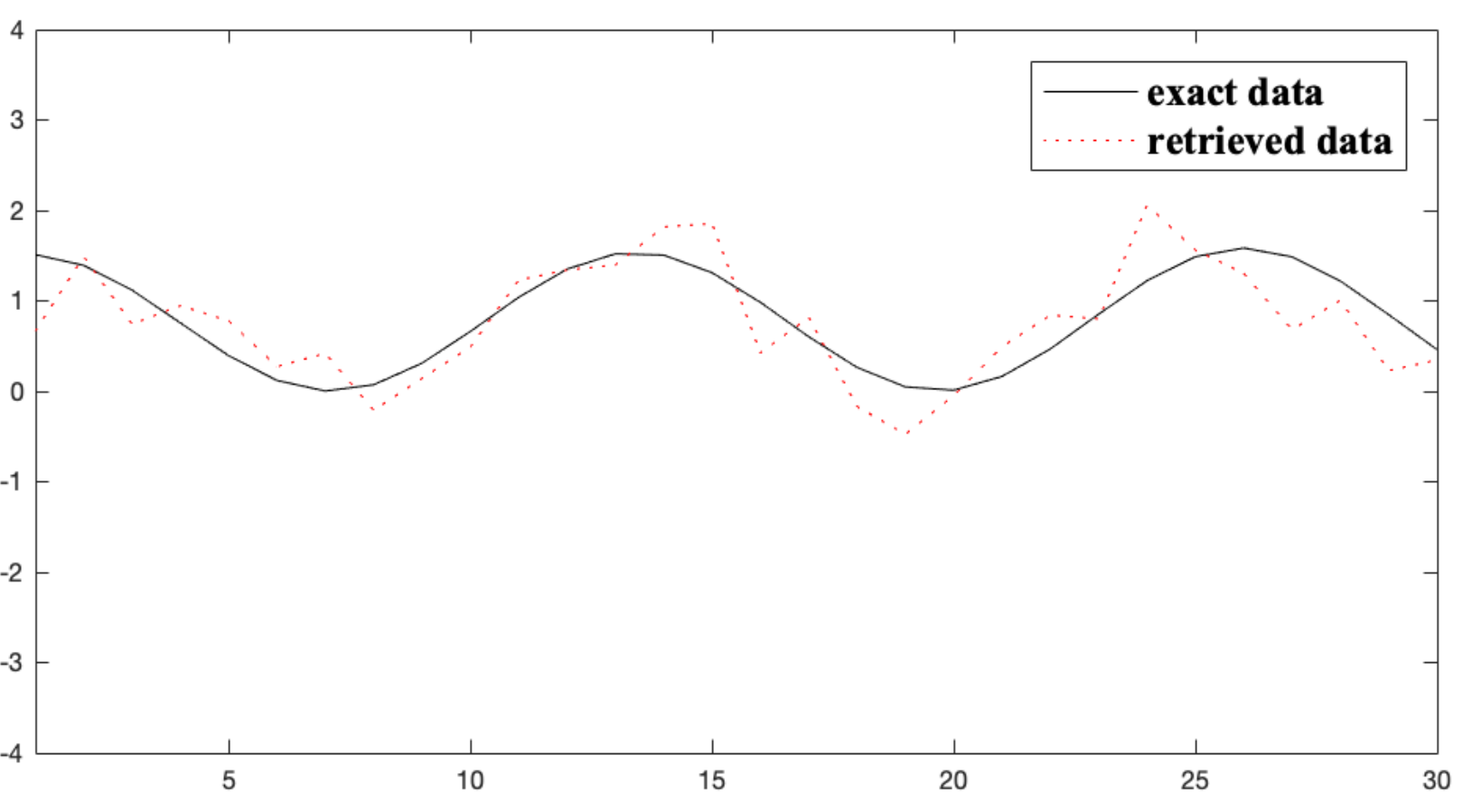}}
\caption{{\bf Example PhaseRetrieval.}\, Phase retrieval for the real part  with relative errors at a fixed direction $\bf\hx_0=(1,0,0)^{\rm T}$.}
\label{error1}
\end{figure}

\begin{figure}[htbp]
  \centering
  \subfigure[\textbf{No noise.}]{
    \includegraphics[height=1.3in,width=1.5in]{pics/PRerrorquare0.pdf}}
  \subfigure[\textbf{0.1 noise.}]{
    \includegraphics[height=1.3in,width=1.5in]{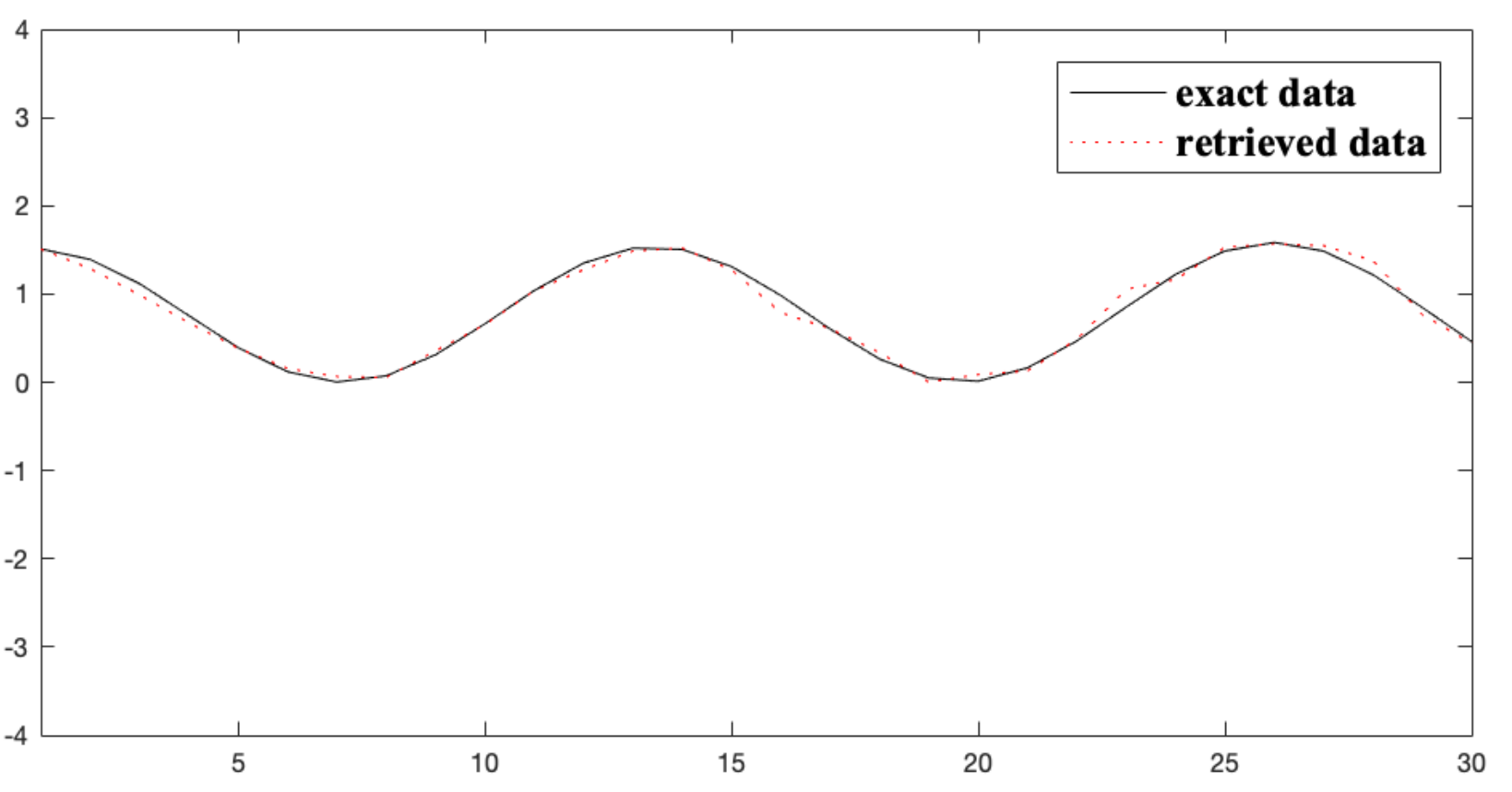}}
  \subfigure[\textbf{0.3 noise.}]{
    \includegraphics[height=1.3in,width=1.5in]{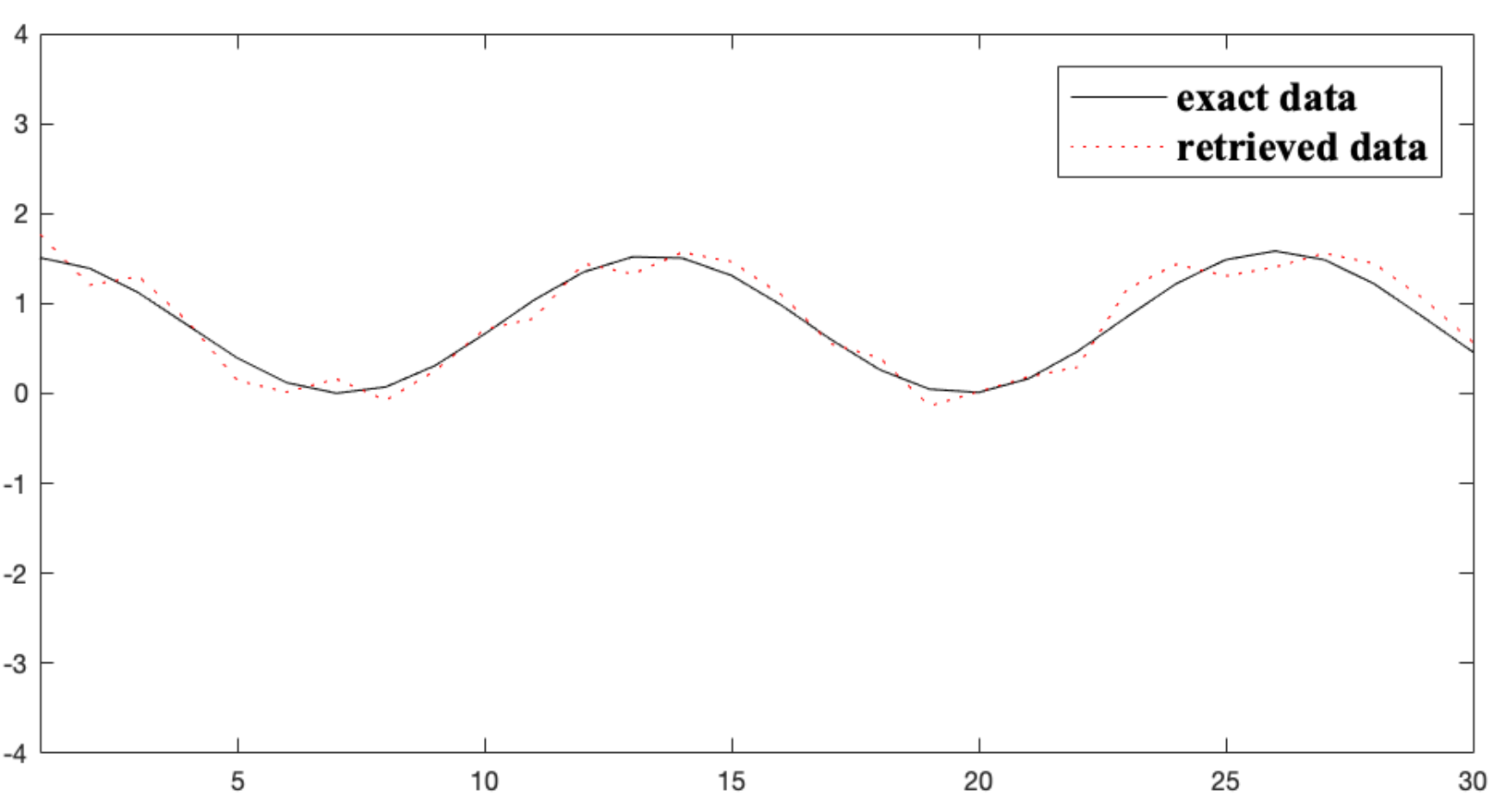}}
\caption{{\bf Example PhaseRetrieval.}\, Phase retrieval for the real part with absolute errors at a fixed direction $\bf\hx_0=(1,0,0)^{\rm T}$.}
\label{error2}
\end{figure}

\subsection{{\bf Strip Reconstruction Scheme Three} for cubic support (${\bf S_1}$)}
In this example, we consider the cubic reconstructions with $1, 2$ and $20$ observation directions. Fig. \ref{square5} gives the results, which is similar to the ones in Fig. \ref{square1}.

\begin{figure}[htbp]
  \centering
  \subfigure[\textbf{One observation direction.}]{
    \includegraphics[width=1.5in]{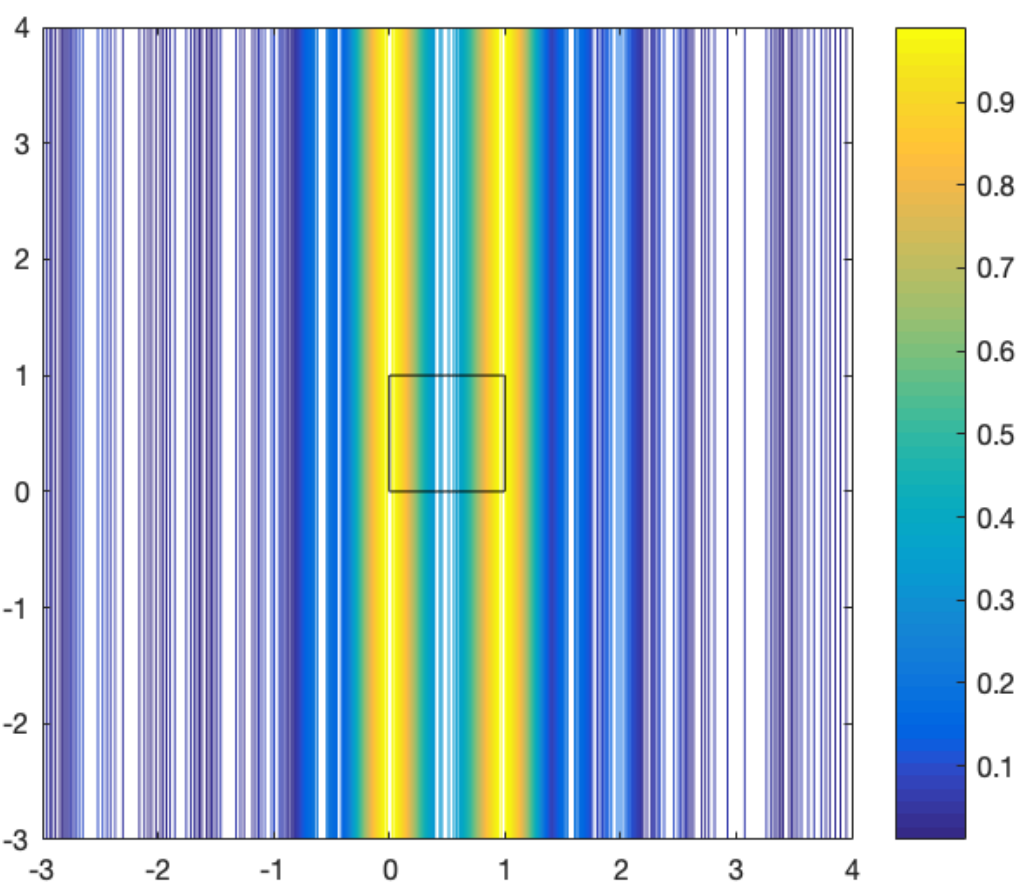}}
  \subfigure[\textbf{Two observation directions.}]{
    \includegraphics[width=1.5in]{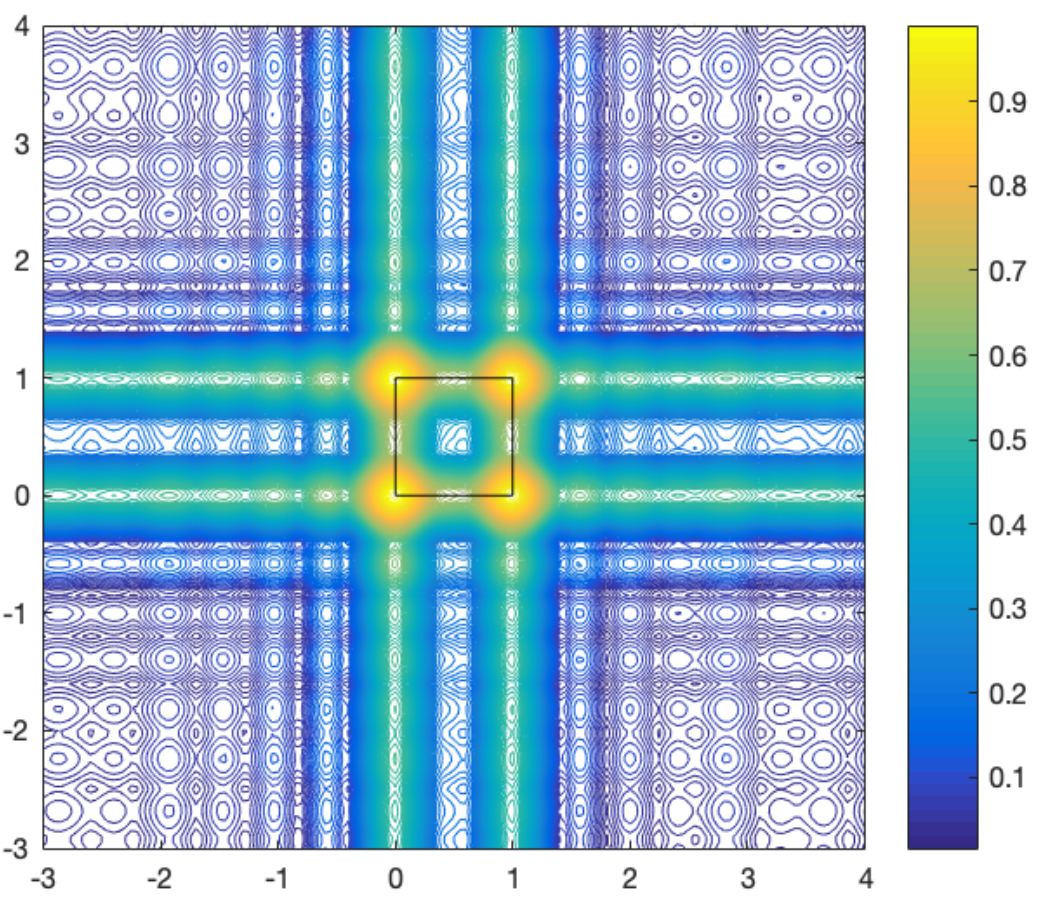}}
  \subfigure[\textbf{Twenty observation directions.}]{
    \includegraphics[width=1.5in]{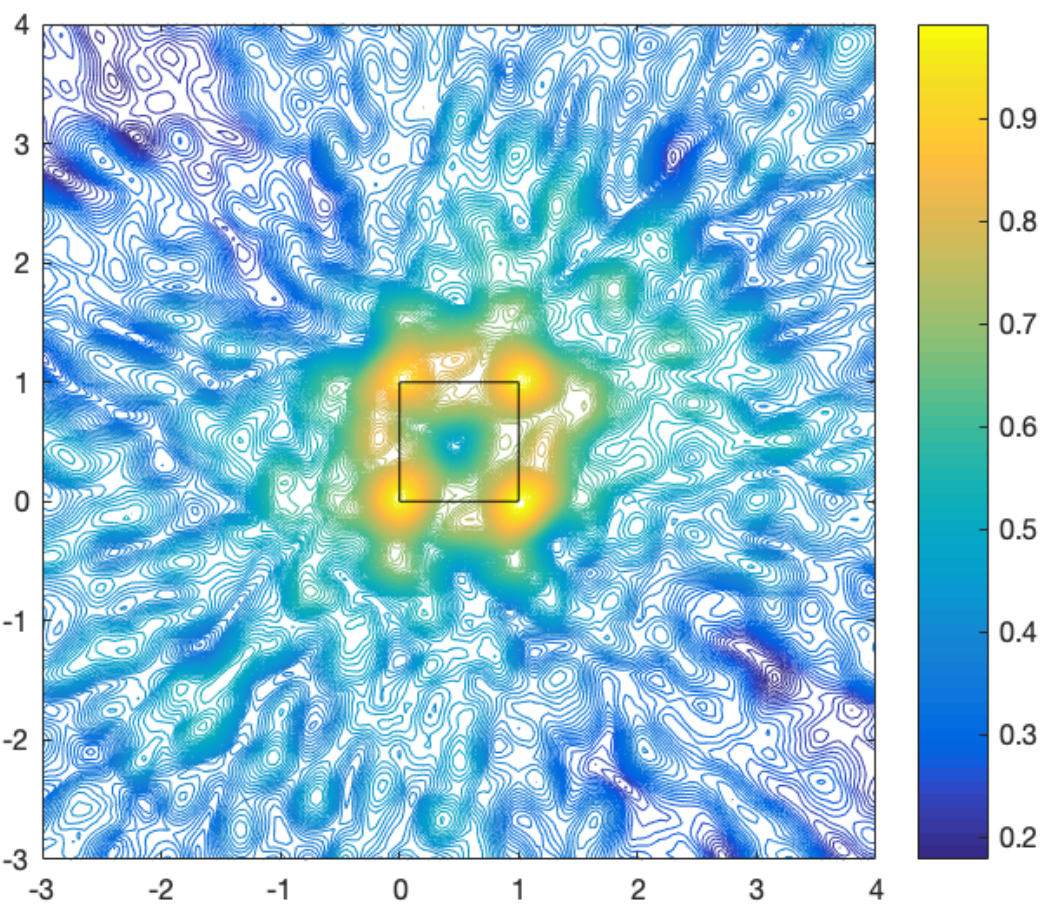}}
\caption{{\bf Strip Reconstruction Scheme Three} with different observation directions for cubic support (${\bf S_1}$).}
\label{square5}
\end{figure}

\subsection{{\bf Strip Reconstruction Scheme Three} for ball support (${\bf S_2}$)}
In this example, we consider the ball reconstruction with $1, 2$ and $20$ observation directions. The phenomenon is similar to the corresponding example in {\bf Strip Reconstruction Scheme One} . We give the results in Fig. \ref{ball5}. The support is
clearly reconstructed in Fig. \ref{ball5}(c).

\begin{figure}[htbp]
  \centering
  \subfigure[\textbf{One observation direction.}]{
    \includegraphics[width=1.5in]{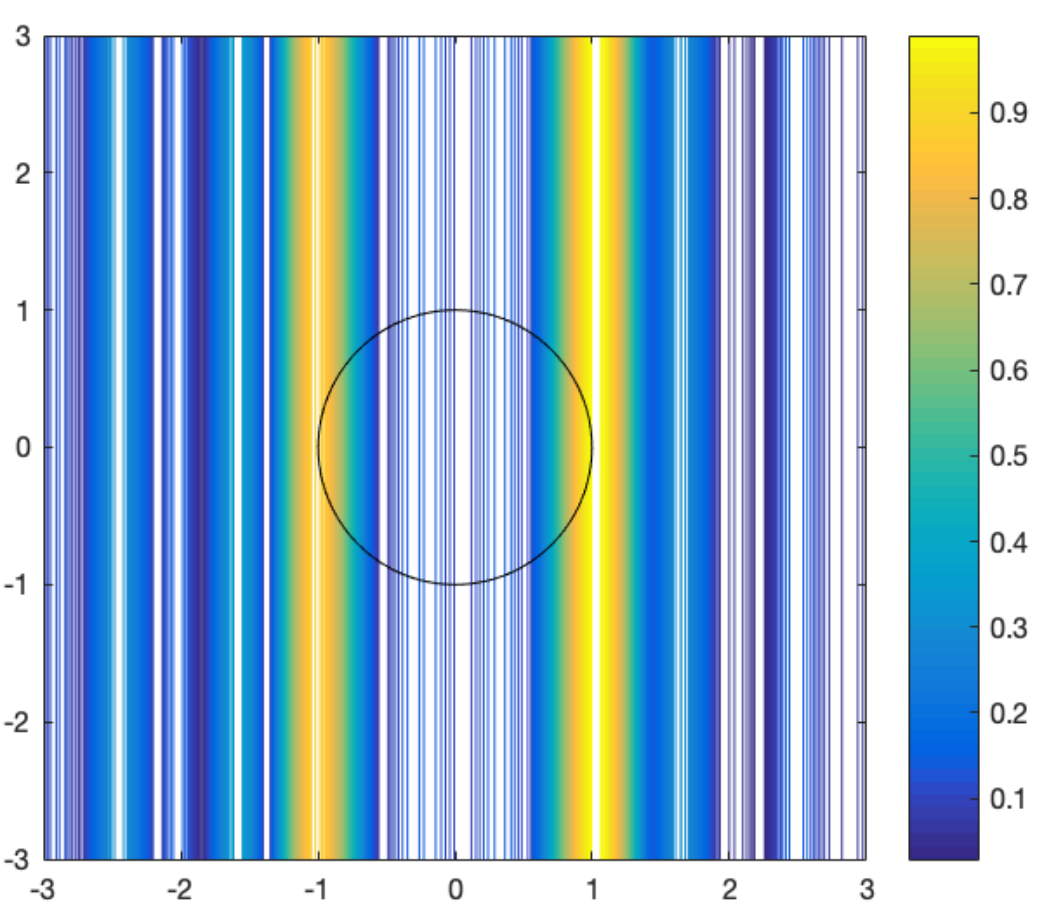}}
  \subfigure[\textbf{Two observation directions.}]{
    \includegraphics[width=1.5in]{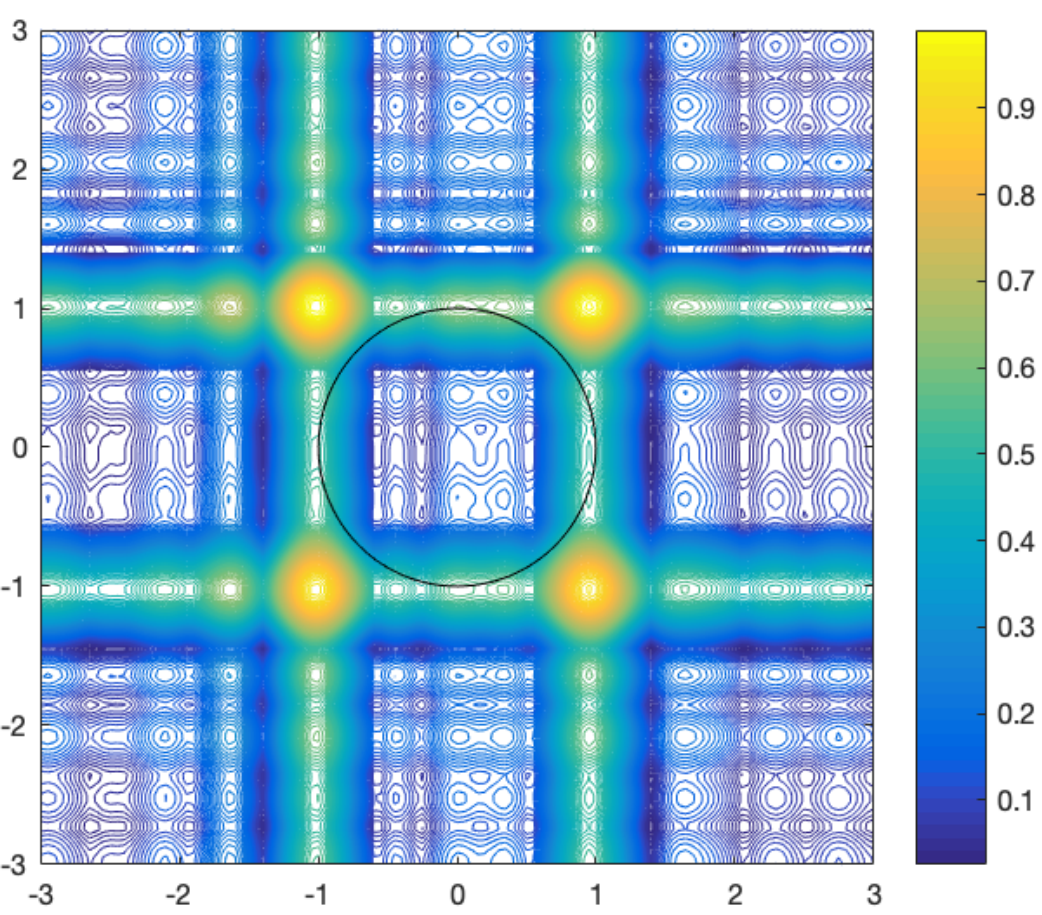}}
  \subfigure[\textbf{Twenty observation directions.}]{
    \includegraphics[width=1.5in]{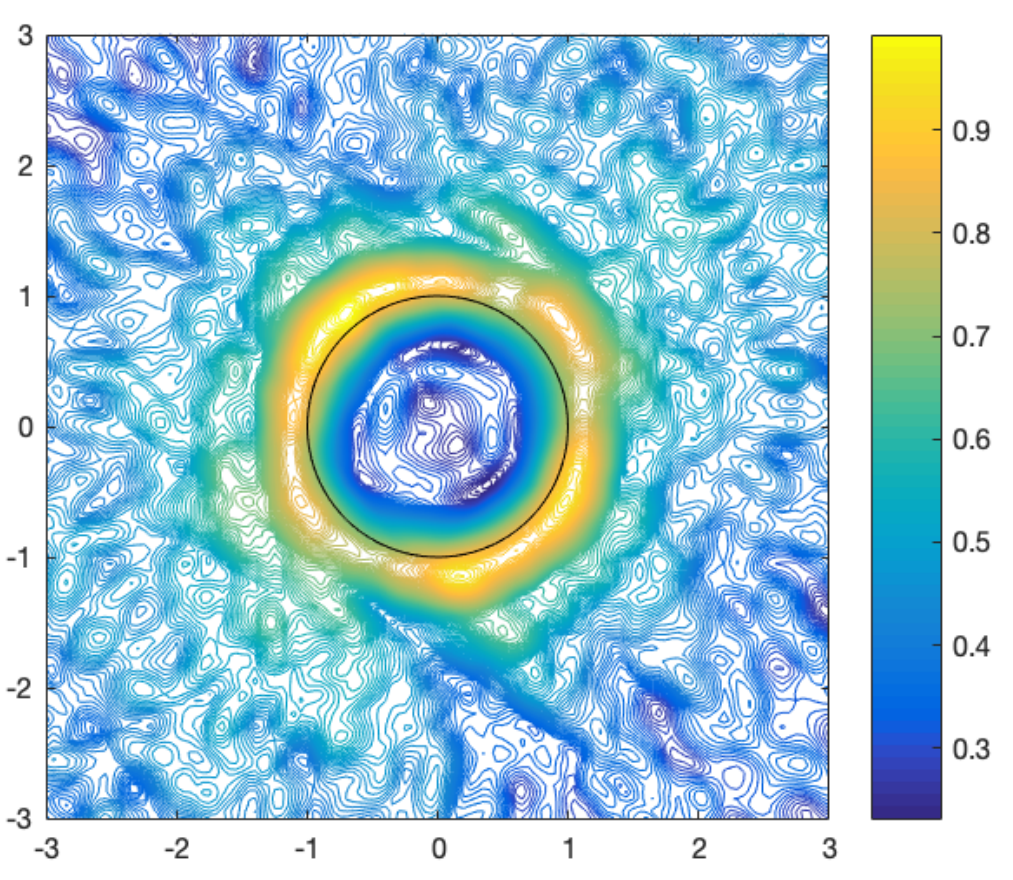}}
\caption{{\bf Strip Reconstruction Scheme Three} with different observation directions for ball support (${\bf S_2}$).}
\label{ball5}
\end{figure}

\subsection{{\bf Strip Reconstruction Scheme Three} for other supports (${\bf S_3}, {\bf S_4}, {\bf S_5}$)}
In this example, we consider $20$ observations for the supports ${\bf S_3}, {\bf S_4}$ and ${\bf S_5}$. In Fig. \ref{S3S42} (a), the cubic and ball are both well constructed. The Fig. \ref{S3S42} (b) gives the L-shaped domain. In Fig. \ref{S52}, we give the ${\bf x-y}$ projection and ${\bf y-z}$ projection of the cuboid.

\begin{figure}[htbp]
  \centering
  \subfigure[\textbf{Cubic+ball.}]{
    \includegraphics[width=1.5in]{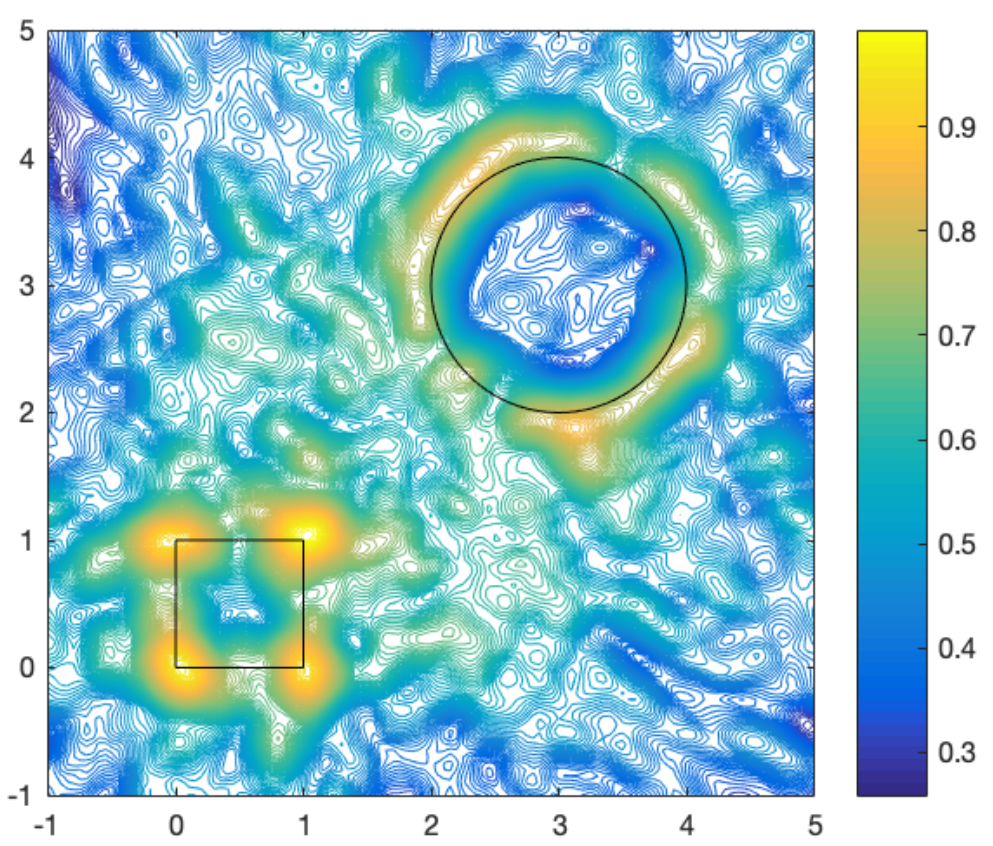}}
  \subfigure[\textbf{L-shaped domain.}]{
    \includegraphics[width=1.5in]{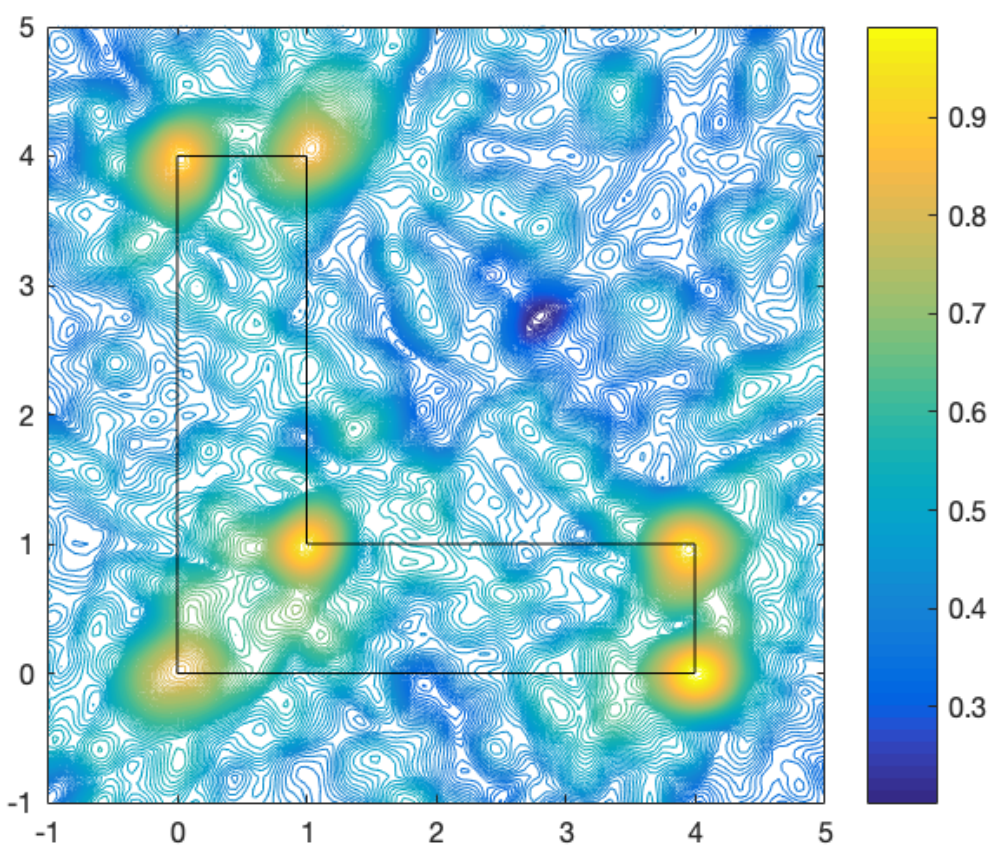}}
\caption{{\bf Strip Reconstruction Scheme Three} with $20$ observation directions for supports ${\bf S_3}$ and ${\bf S_4}$.}
\label{S3S42}
\end{figure}

\begin{figure}[htbp]
  \centering
  \subfigure[\textbf{$x-y$ projection.}]{
    \includegraphics[width=1.5in]{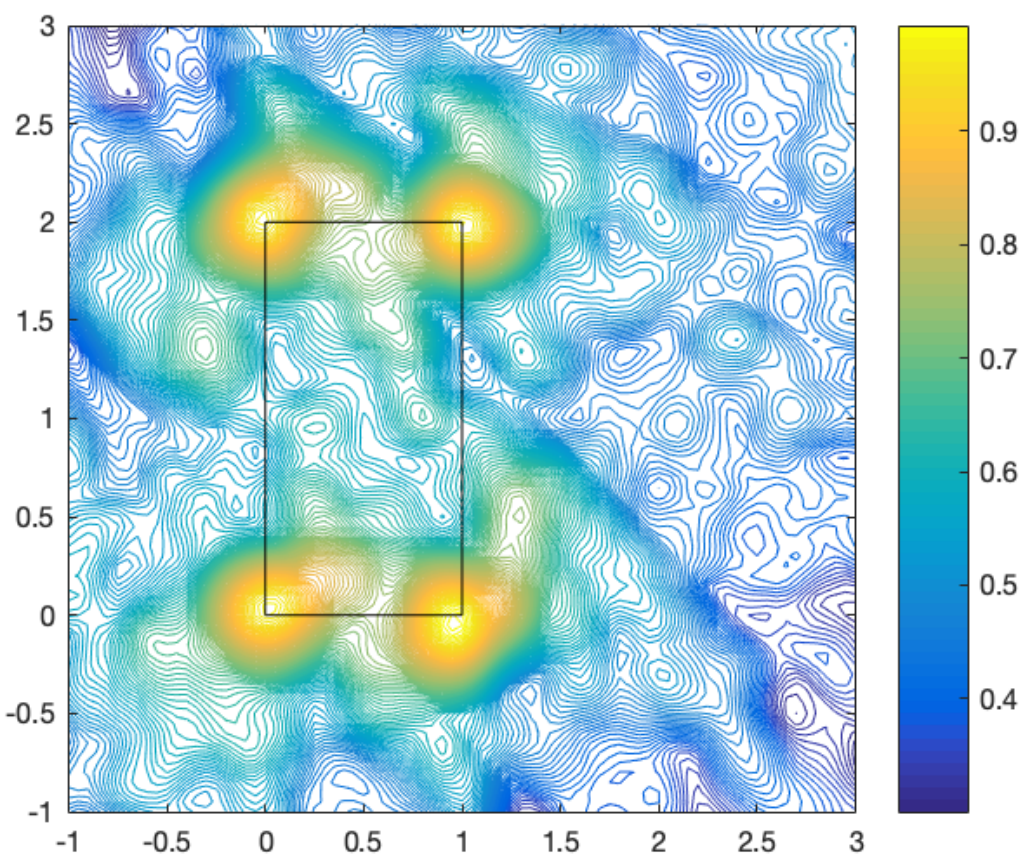}}
  \subfigure[\textbf{$y-z$ projection.}]{
    \includegraphics[width=1.5in]{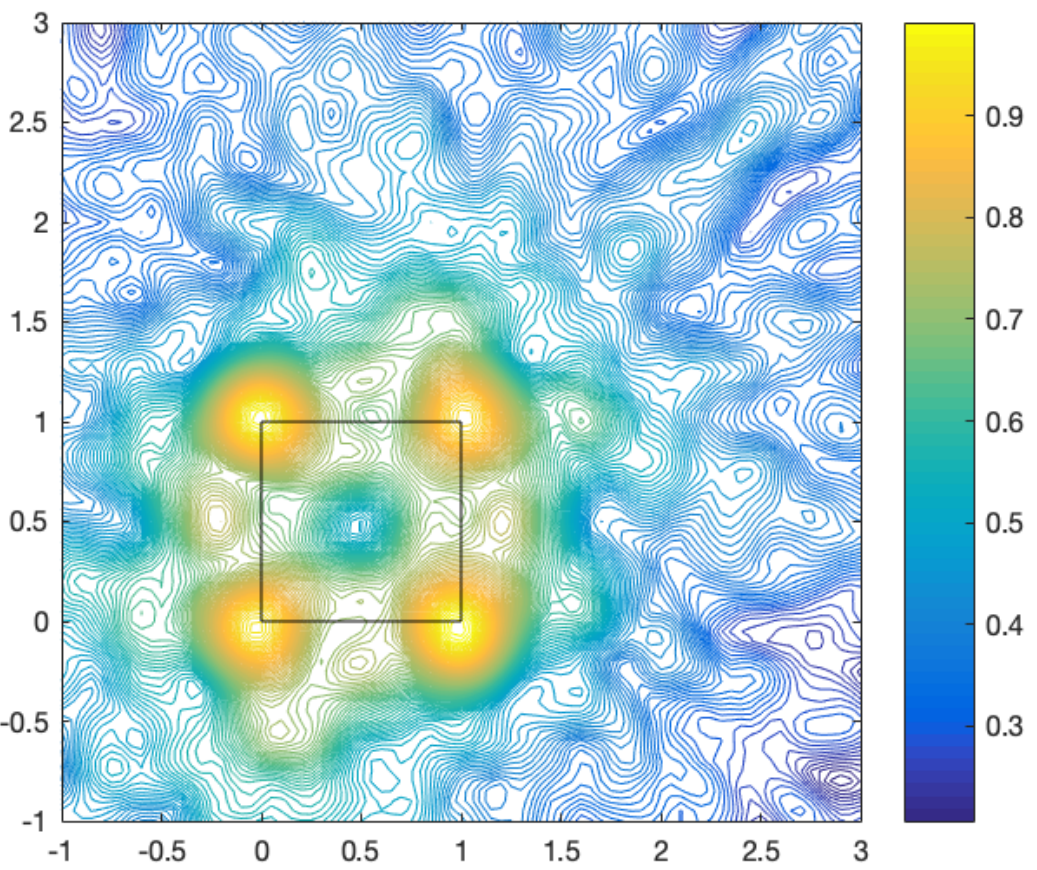}}
\caption{{\bf Strip Reconstruction Scheme Three} with $20$ observation directions for cuboid support ${\bf S_5}$.}
\label{S52}
\end{figure}

\section*{Acknowledgement}
The research of X. Ji is partially supported by the NNSF of China with Grant Nos. 11271018 and 91630313,
and National Centre for Mathematics and Interdisciplinary Sciences, CAS.
The research of X. Liu is supported by the NNSF of China under grant 11571355 and the Youth Innovation Promotion Association, CAS.

\end{document}